\newif\ifabstract
\abstracttrue
 \abstractfalse 
\newif\iffull
\ifabstract \fullfalse \else \fulltrue \fi

\documentclass[11pt]{article}

\usepackage[utf8]{inputenc} 
\usepackage[T1]{fontenc}    
\usepackage{hyperref}       
\usepackage{url}            
\usepackage{booktabs}       
\usepackage{amsfonts}       
\usepackage{nicefrac}       
\usepackage{microtype}      

\usepackage[algo2e,inoutnumbered,linesnumbered,algoruled,vlined]{algorithm2e}

\usepackage[active]{srcltx}

\usepackage{wrapfig}
\usepackage{mathtools}

\usepackage{subfigure}

\usepackage{subfigure}
\usepackage{amsmath}
\usepackage{cleveref}
\usepackage{amssymb}
\usepackage{comment}
\usepackage{graphicx}
\usepackage{cancel}
\usepackage{amsthm}
\usepackage[square,sort,comma,numbers]{natbib}
\usepackage{dsfont}
\usepackage{xcolor}

\usepackage{enumitem}
\usepackage{lmodern}

\newtheorem{remark}{Remark}
\newtheorem{lemma}{Lemma}

\DeclareMathOperator*{\argmin}{arg\min}

\newcommand{\sas}{{\cal S} \alpha {\cal S} }

\usepackage{cleveref}

\newtheorem{assumption}{\textbf{H}\hspace{-3pt}}
\Crefname{assumption}{\textbf{H}\hspace{-3pt}}{\textbf{H}\hspace{-3pt}}
\crefname{assumption}{\textbf{H}}{\textbf{H}}

\graphicspath{{./},{./figures/}}

\begin{document}

\title{Non-Asymptotic Analysis of Fractional Langevin Monte Carlo for \\ Non-Convex Optimization}
\author{Thanh Huy Nguyen, Umut \c{S}im\c{s}ekli, Ga\"{e}l Richard  \vspace{3pt} \\
{\small LTCI, T\'{e}l\'{e}com Paristech, Universit\'{e} Paris-Saclay, 75013, Paris, France }}
\date{}

\maketitle

\noindent

\begin{abstract}
Recent studies on diffusion-based sampling methods have shown that Langevin Monte Carlo (LMC) algorithms can be beneficial for non-convex optimization, and rigorous theoretical guarantees have been proven for both asymptotic and finite-time regimes. Algorithmically, LMC-based algorithms resemble the well-known gradient descent (GD) algorithm, where the GD recursion is perturbed by an additive Gaussian noise whose variance has a particular form. Fractional Langevin Monte Carlo (FLMC) is a recently proposed extension of LMC, where the Gaussian noise is replaced by a heavy-tailed $\alpha$-stable noise. As opposed to its Gaussian counterpart, these heavy-tailed perturbations can incur large jumps and it has been empirically demonstrated that the choice of $\alpha$-stable noise can provide several advantages in modern machine learning problems, both in optimization and sampling contexts. However, as opposed to LMC, only asymptotic convergence properties of FLMC have been yet established. In this study, we analyze the non-asymptotic behavior of FLMC for non-convex optimization and prove finite-time bounds for its expected suboptimality. Our results show that the weak-error of FLMC increases faster than LMC, which suggests using smaller step-sizes in FLMC. We finally extend our results to the case where the exact gradients are replaced by stochastic gradients and show that similar results hold in this setting as well. 

\end{abstract}

\newtheorem{corollary}{Corollary}
\newtheorem{theorem}{Theorem}

\tableofcontents


\section{Introduction }


Diffusion-based Markov Chain Monte Carlo (MCMC) algorithms have become increasingly popular in the recent years due to their nice scalability properties and theoretical guarantees. The main aim in these approaches is to generate samples from a distribution which is only accessible by its unnormalized density function.

One of the most popular approaches in this field is based on the so-called Langevin diffusion, which is described by the following stochastic differential equation (SDE):
\begin{align}
\mathrm{d}X(t)=- \nabla f(X(t))\mathrm{d}t + \sqrt{2/\beta}\>\mathrm{d} \mathrm{B}(t),\,\,\,t\geq0, \label{eqn:langevin_sde}
\end{align}
where $X(t) \in \mathbb{R}^d$, $f$ is a smooth function which is often non-convex, $\beta \in \mathbb{R}_+$ is called the `inverse temperature' parameter, and $\mathrm{B}(t)$ is the standard Brownian motion in $\mathbb{R}^d$.

Under some regularity conditions on $f$, one can show that the Markov process $(X_t)_{t \geq 0}$, i.e.\ the solution of the SDE \eqref{eqn:langevin_sde}, is ergodic with its unique invariant measure $\pi$, whose density is proportional to $\exp(-\beta f(x))$ \cite{Roberts03}. An important feature of this measure is that, when $\beta$ goes to infinity, its density concentrates around the global minimum $x^\star \triangleq \argmin_{x\in \mathbb{R}^d} f(x)$ \cite{hwang1980laplace,gelfand1991recursive}.
This property implies that, if we could simulate \eqref{eqn:langevin_sde} for large enough $\beta$ and $t$, the simulated state $X(t)$ would be close to $x^\star$.

This connection between diffusions and optimization, motivates simulating \eqref{eqn:langevin_sde} in discrete-time in order to obtain `almost global optimizers'. If we use a first-order Euler-Maruyama discretization, we obtain a `tempered' version of the well-known Unadjusted Langevin Algorithm (ULA) \cite{Roberts03}:
\begin{align}
W^{k+1}_{\text{ULA}} = W^{k}_{\text{ULA}} - \eta \nabla f(W^{k}_{\text{ULA}}) + \sqrt{\frac{2\eta}{\beta}} \Delta B_{k+1}, \label{eqn:ula}
\end{align}
where $k \in \mathbb{N}_+$ denotes the iterations, $\eta$ denotes the step-size, and $(\Delta B_{n})_n$ is a sequence of independent and identically-distributed (i.i.d.) standard Gaussian random variables. When $\beta =1$, we obtain the classical ULA, which is mainly used for Bayesian posterior sampling. Theoretical properties of the classical ULA have been extensively studied \cite{Roberts03,lamberton2003recursive,durmus2015non,durmus2016high,dalalyan2017theoretical}.

When $\beta \gg 1$, the algorithm is called tempered and becomes more suitable for optimization. Indeed, one can observe that the noise term $\Delta B_k$ in \eqref{eqn:ula} becomes less dominant, and the overall algorithm can be seen as a `perturbed' version of the gradient descent (GD) algorithm. The connection between ULA and GD has been recently established in \cite{dalalyan2017further} for strongly convex $f$. Moreover, \cite{raginsky17a} and \cite{xu2018global} proved non-asymptotic guarantees for this perturbed scheme\footnote{The results given in \cite{raginsky17a} are more general in the sense that they are proved for the Stochastic Gradient Langevin Dynamics (SGLD) algorithm \cite{WelTeh2011a}, which is obtained by replacing the gradients in \eqref{eqn:ula} with stochastic gradients.}. Their results showed that, even in non-convex settings, the algorithm is guaranteed to escape from local minima and converge near the global minimizer. These results were extended in \cite{zhang17b} and \cite{tzen2018local}, which showed that the iterates converge near a local minimum in polynomial time and stay there for an exponential time. Recently, the guarantees for ULA were further extended to second-order Langevin dynamics \cite{gao2018global,gao2018breaking}.

\begin{figure*}[h]
\centering
\includegraphics[width=0.49\columnwidth]{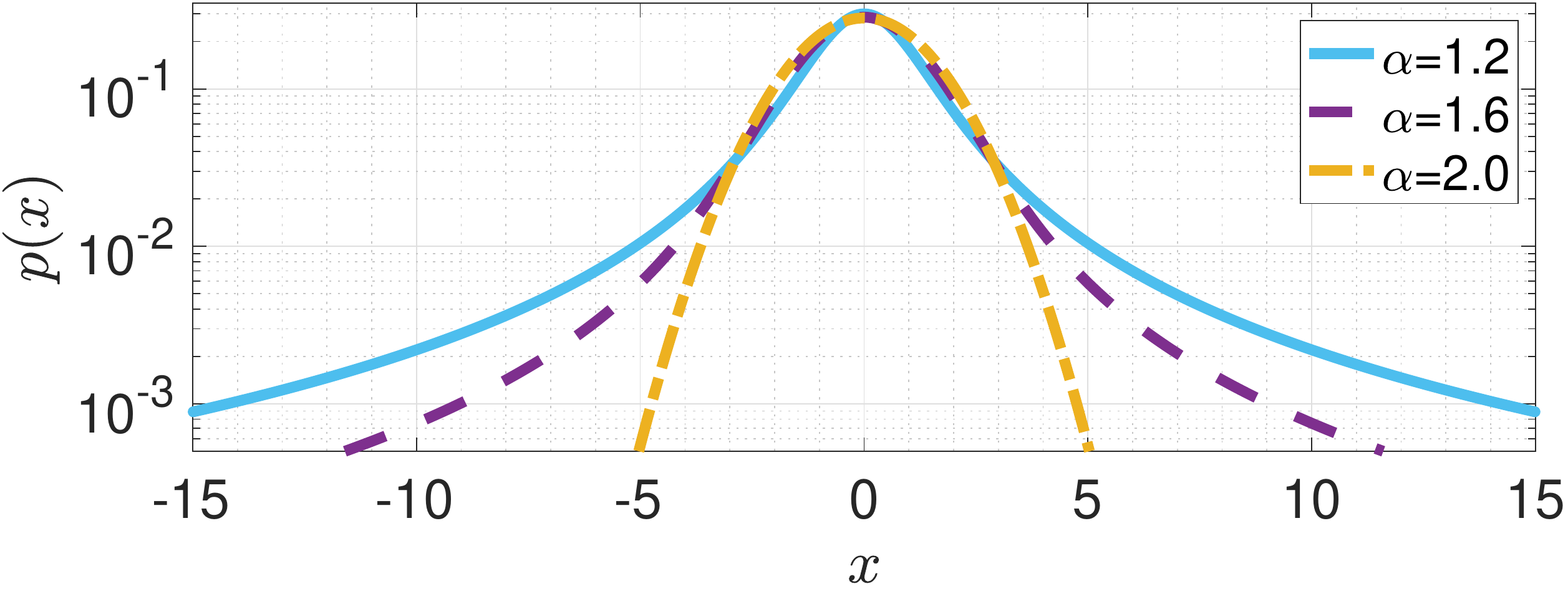} 
\hfill  \includegraphics[width=0.49\columnwidth]{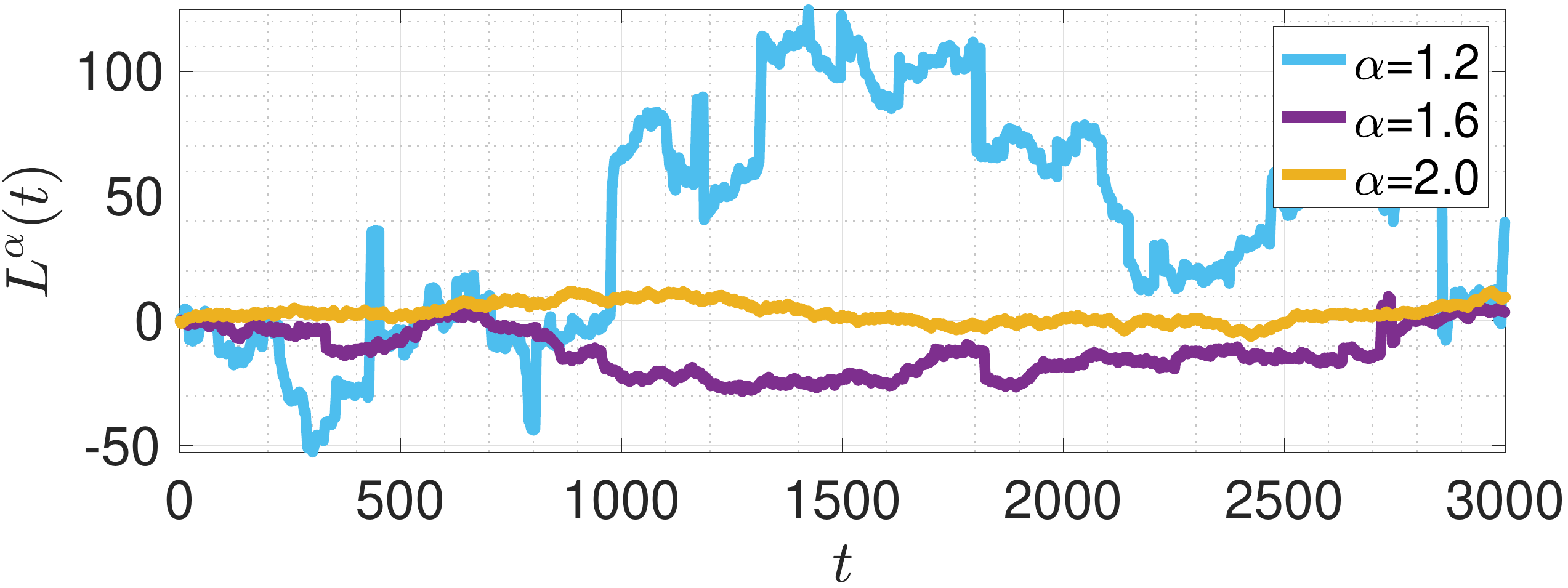}  
\vspace{-12pt}
\caption{Illustration of the density function of the symmetric $\alpha$-stable ($\sas$) distribution (left) and the $\alpha$-stable L\'{e}vy motion (right). As $\alpha$ gets smaller, $\sas$ becomes heavier-tailed and consequently, $L^\alpha(t)$ incurs larger jumps.}  
\vspace{-12pt}
\label{fig:stablepdf_motion}
\end{figure*}

Another line of research has extended Langevin Monte Carlo by replacing the Brownian motion with a motion which can incur `jumps' (i.e.\ discontinuities), such as the \emph{$\alpha$-stable L\'{e}vy Motion}  (see Figure~\ref{fig:stablepdf_motion}) \cite{simsekli17a,ye_sfhmc}. Coined under the name of Fractional Langevin Monte Carlo (FLMC) methods, these approaches are motivated by the statistical physics origins of the Langevin equation \eqref{eqn:langevin_sde}. In such a context, the Langevin equation aims to model the position of a small particle that is under the influence of a force, which has a deterministic and a stochastic part. If we assume that the stochastic part of this force is a sum of many i.i.d.\ random variables with finite variance, then by the central limit theorem (CLT), we can assume that their sum follows a Gaussian distribution, which justifies the Brownian motion in \eqref{eqn:langevin_sde}. 

The main idea in FLMC is to relax the finite variance assumption and allow the random pulses to have infinite variance. In such a case, the classical CLT will not hold; however, the \emph{extended} CLT \cite{paul1937theorie} will still be valid: the law of the sum of the pulses converges to an \emph{$\alpha$-stable} distribution, a family of `heavy-tailed' distributions that contains the Gaussian distribution as a special case. Then, by using a similar argument to the previous case, we can replace the Brownian motion with the $\alpha$-stable L\'{e}vy Motion \cite{yanovsky2000levy}, whose increments are $\alpha$-stable distributed. 

Based on an SDE driven by an $\alpha$-stable L\'{e}vy Motion, \cite{simsekli17a} proposed the following iterative scheme that is referred to as Fractional Langevin Algorithm (FLA):
\begin{align}
W^{k+1}_{\text{FLA}} = W^{k}_{\text{FLA}} - \eta c_\alpha \nabla f(W^{k}_{\text{FLA}}) + \Bigl(\frac{\eta}{\beta}\Bigr)^{\frac1{\alpha}} \Delta L^{\alpha}_{k+1}, \label{eqn:fla}
\end{align}
where $\alpha \in (1,2]$ is called the characteristic index, $c_\alpha$ is a known constant, and $\{\Delta L^{\alpha}_{k}\}_{k\in \mathbb{N}_+}$ is a sequence of $\alpha$-stable distributed random variables. As we will detail in Section~\ref{sec:prel}, FLA coincides with ULA when $\alpha =2$.
Recently, \cite{ye_sfhmc} extended FLA to Hamiltonian dynamics. The experimental results in \cite{simsekli17a} and \cite{ye_sfhmc} showed that the use of the heavy-tailed increments can provide advantages in multi-modal settings, robustness to algorithm parameters. \cite{ye_sfhmc} further illustrated that in an optimization context their algorithm achieves better generalization in deep neural networks.

Even though asymptotic convergence properties of FLMC were established for decreasing step-sizes in \cite{simsekli17a,panloup2008recursive}, these results do not explain the behavior of the algorithm for finite number of iterations. Besides, in practice, using a constant step-size often yields better performance \cite{baker2017sgmcmc}, a situation which cannot be handled by the existing theory.

\subsection{Overview of the main result}

In this study, we analyze the non-asymptotic behavior of FLA for non-convex optimization. In particular, we analyze the expected suboptimality $\mathbb{E}[f(W^{k}_{\text{FLA}}) - f^\star]$, where $f^\star \triangleq f(x^\star)$. As we will describe in detail in Section~\ref{sec:proof_overview}, we decompose this suboptimality into four different terms, and we bound each of those terms one by one. Due to the choice of the $\alpha$-stable L\'{e}vy motion, the standard tools for analyzing SDEs driven by a Brownian motion are not available for our use, and therefore, we cannot use the proof strategies developed for ULA as they are (such as \cite{raginsky17a,xu2018global,erdogdu2018global}). Instead, we follow an alternative path, where we first relate the expected discrepancies to Wasserstein distance of fractional orders, and then, inspired by \cite{gairing2018transportation}, we prove a result that expresses the Wasserstein distance between the laws of two SDEs (driven by $\alpha$-stable L\'{e}vy motion) in terms of their drift functions. 


Informally, we show that the expected suboptimality $\mathbb{E}[f(W^{k}_{\text{FLA}}) - f^\star]$ is bounded by a sum of four terms, summarized as follows:
\begin{align*}
\mathbb{E}[f(W^{k}_{\text{FLA}}) - f^\star] &\leq \mathcal{A}_1 + \mathcal{A}_2 + \mathcal{A}_3 + \mathcal{A}_4, 
\end{align*}
where
\begin{align*}
\mathcal{A}_1 &= \mathcal{O} \Bigl( k^{1+\max\{\frac{1}{q},\gamma+\frac{\gamma}{q}\}}\eta^{\frac{1}{q}} \Bigr), \quad\mathcal{A}_2 = \mathcal{O} \Bigl( \frac{k^{1+\max\{\frac{1}{q},\gamma+\frac{\gamma}{q}\}}\eta^{\frac{1}{q}+\frac{\gamma}{\alpha q}}d}{\beta^{\frac{(q-1)\gamma}{\alpha q}}} \Bigr),\\
\mathcal{A}_3  &= \mathcal{O} \Bigl(\beta+d \Bigr) \exp\Bigl(- \frac{\lambda_* k\eta}{\beta}  \Bigr), \quad\mathcal{A}_4 = \mathcal{O} \Bigl(\frac1{\beta^{\gamma+1}} + \frac{d}{\beta} \log(\beta+1) \Bigr).
\end{align*}
Here $\gamma \in (0,1)$ is the H\"{o}lder exponent of the gradients of $f$, and $q \in (1,\alpha)$, $\lambda_*>0$ are some constants. This result has the following implications. For any $\varepsilon>0$,
\begin{enumerate}[itemsep = 1pt, topsep=0pt,leftmargin=*,align=left]
\item If $\frac{1}{q}>\gamma+\frac{\gamma}{q}$ and 
$k \simeq \varepsilon^{-1} \quad \text{and} \quad \eta < \varepsilon^{2q+1}$, 
then ${\cal A}_1$ scales as $C\varepsilon$ and ${\cal A}_2$ scales as $\varepsilon\mathrm{Poly}(\beta,d)$.

\item If $\frac{1}{q}\leq\gamma+\frac{\gamma}{q}$ and
$k \simeq \varepsilon^{-1} \quad \text{and} \quad \eta < \varepsilon^{2q+\gamma+\gamma q}$, 
then ${\cal A}_1$ scales as $C\varepsilon$ and ${\cal A}_2$ scales as $\varepsilon\mathrm{Poly}(\beta,d)$. 
%
\item If we choose 
$k\eta > \frac{\beta}{\lambda_*} \log\Bigl(\frac1{\varepsilon}\Bigr)$,
then ${\cal A}_3$ scales as $\varepsilon\mathrm{Poly}(\beta,d)$. 
\end{enumerate}
Here, $\mathrm{Poly}(\ldots)$ denotes a formal polynomial, i.e., an expression containing the real-ordered exponents of the variables, coefficients, and only the operations of addition, subtraction, and multiplication.

In Section~\ref{sec:extensions}, we extend our results in two directions: (i) obtaining guarantees for Bayesian posterior sampling and (ii) non-convex optimization where exact gradients are replaced with stochastic gradients. Our results imply that, in the context of global optimization, the error induced by FLA has a worse dependency on $k$ and $\eta$, as compared to ULA. This suggests that one should use smaller step-sizes in FLA.

\newtheorem{definition}{Definition}

\section{Technical Background and Preliminaries}

\label{sec:prel}


\subsection{Notations and basic definitions}

In this section, we will define the basic quantities that will be used throughout the paper. 
%
We use $<\cdot,\cdot>$ to denote the inner product between two vectors, $\Vert\cdot\Vert$ denotes the Euclidean norm, $\mathbb{E}_{\omega}[\cdot]$ denotes the expectation with respect to the random variable $\omega$, and $\mathbb{E}[\cdot]$ denotes the expectation with respect to all the random sources. We will use the Wasserstein metric to quantify the distance between two probability measures.
\begin{definition}[Wasserstein distance] Let $\mu$ and $\nu$ be two probability measures. For $\lambda\geq1$, we define the $\lambda$-Wasserstein distance between $\mu$ and $\nu$ as follows:
\begin{align*}
\mathcal{W}_{\lambda}(\mu,\nu)\triangleq(\inf\{\mathbb{E}\Vert V-W\Vert^{\lambda}: V\sim\mu,W\sim\nu\})^{1/\lambda},
\end{align*}
where the infimum is taken over all the couplings of $\mu$ and $\nu$ (i.e.\ the joint probability distributions whose marginal distributions are $\mu$ and $\nu$). 
\end{definition}
From now on, we will denote $W^{k}_{\text{FLA}}$ as $W^{k}$ for notational simplicity.
All the proofs are given in the supplementary document.

\subsection{$\alpha$-Stable Distributions and $\alpha$-Stable L\'{e}vy Motion}\label{subsec:levyProcess}

\begin{definition}[Symmetric $\alpha$-stable random variables]\label{def:StableDistribution}
The $\alpha$-stable distribution appears as the limiting distribution in the generalized CLT \cite{samorodnitsky1994stable}. A scalar random variable $X \in \mathbb{R}$ is called symmetric $\alpha$-stable if its characteristic function has the following form:
\begin{align*}
\mathbb{E}[e^{i \omega X}]=\exp(-\sigma\vert \omega\vert^{\alpha}) 
\end{align*}
where $\alpha \in (0,2]$ and $\sigma>0$. We denote $X\sim \sas(\sigma)$. 
\end{definition}

The parameter $\alpha$ is called the \emph{characteristic index} or the \emph{tail index}, since it determines the tail behavior of the distribution. Perhaps the most important special case of symmetric $\alpha$-stable distributions is the Gaussian distribution: $\sas(\sigma) = {\cal N}(0,2\sigma^2)$ when $\alpha =2$. As we decrease $\alpha$, the distribution becomes \emph{heavier-tailed}. Moreover, when $X \sim \sas(\sigma)$, the moment $\mathbb{E}[|X|^p]$ is finite if and only if $p < \alpha$. This implies that the distribution has infinite variance (i.e.\ the variance diverges) whenever $\alpha \neq 2$.

\begin{definition}[Symmetric $\alpha$-stable L\'evy motion]\label{def:symAlphaLevy}
A scalar symmetric $\alpha$-stable L\'evy motion $L^{\alpha}(t)$, with $0<\alpha\leq 2$, is a stochastic process satisfying the following properties:

\begin{enumerate}[label=(\roman*),noitemsep,topsep=0pt,leftmargin=*,align=left]
\item $L^{\alpha}(0) = 0$, almost surely.
\item Independent increments: for $0\leq t_1<\ldots<t_n$, the random variables  $L^{\alpha}(t_2)-L^{\alpha}(t_1)$,..., $L^{\alpha}(t_{n})-L^{\alpha}(t_{n-1})$ are independent.
\item Stationary increments: for all $0\leq s< t$, the random variables $L^{\alpha}(t)-L^{\alpha}(s)$ and $L^{\alpha}(t-s)$ have the same distribution as $\sas((t-s)^{1/\alpha})$.
\item Continuity in probability: for any $\delta >0$ and $s\geq 0$, $\mathbb{P}(\vert L^{\alpha}(s) - L^{\alpha}(t)\vert>\delta)\rightarrow 0$, as $t\rightarrow s$.
\end{enumerate}
\end{definition}
We illustrate $\sas$ and $L^\alpha(t)$ in Figure~\ref{fig:stablepdf_motion}. In the rest of the paper, $L^{\alpha}(t)$ will denote a $d$-dimensional L\'evy process whose components are independent scalar symmetric $\alpha$-stable L\'evy motions as defined in \Cref{def:symAlphaLevy}.

\subsection{Fractional Langevin Monte Carlo}


The FLMC framework is based on a L\'{e}vy-driven SDE, that is defined as follows:
\begin{align}\label{eqn:fracLangevin}
\text{d}X(t)=\Psi(X(t-),\alpha)\text{d}t + (1/{\beta})^{1/\alpha}\text{d}L^{\alpha}(t) 
\end{align}
where $X(t-)$ denotes the \emph{left limit} of the process at time $t$, $L^{\alpha}(t)$ denotes the $d$-dimensional L\'evy motion as described in \Cref{subsec:levyProcess}.
FLMC is built up on the following result:
\begin{theorem}[\cite{simsekli17a}]
\label{thm:flmc_org}
Consider the SDE \eqref{eqn:fracLangevin} in the case $d=1$, $\beta =1$, and $\alpha \in (1,2]$, where the drift $\Psi$ is defined as follows:
\begin{align}
\label{eqn:flmc_drift_old}\Psi(x,\alpha) \triangleq -\frac{\mathcal{D}^{\alpha-2}\Big(\phi(x)\frac{\partial f(x)}{\partial x}\Big)\Big)}{\phi(x)}.
\end{align}
where $\mathcal{D}$ denotes the fractional Riesz derivative and is defined as follows for a function $u$:
\begin{align*}
\mathcal{D}^\gamma u(x)\triangleq\mathcal{F}^{-1}\{\vert\omega\vert^\gamma\hat{u}(\omega)\},
\end{align*}
Here, $\mathcal{F}$ denotes the Fourier transform and $\hat{u}\triangleq\mathcal{F}(u)$.
Then, $\pi$ is an invariant measure of the Markov process $(X(t))_{t\geq0}$ that is a solution of the SDE given by \eqref{eqn:fracLangevin}. 
\end{theorem}
This theorem states that if the drift \eqref{eqn:flmc_drift_old} can be computed, then the sample paths of \eqref{eqn:fracLangevin} can be considered as samples drawn from $\pi$. However, computing \eqref{eqn:flmc_drift_old} is in general not tractable, therefore one needs to approximate it for computational purposes.
If we use the alternative definition of the Riesz derivative given by \cite{ortigueira2006riesz}, we can approximate the drift as follows \cite{simsekli17a,ye_sfhmc}: 
\begin{align*}
-\frac{\mathcal{D}^{\alpha-2}\Big(\phi(x)\frac{\partial f(x)}{\partial x}\Big)\Big)}{\phi(x)}\approx -c_\alpha\frac{\partial f(x)}{\partial x},
\end{align*}
where $c_{\alpha}\triangleq\Gamma(\alpha-1)/\Gamma(\alpha/2)^2$ and $\Gamma$ denotes the Gamma function. With this choice of approximation, in the $d$-dimensional case we obtain FLA, as given in \eqref{eqn:fla}. We can observe that, when $\alpha=2$, \eqref{eqn:fracLangevin} becomes the Langevin equation \eqref{eqn:langevin_sde} and FLA becomes ULA. 

\section{Assumptions and the Main Result}



We start by defining three different stochastic processes $X_1(t)$, $X_2(t)$, and $X_3(t)$, which will be the main constructs in our analysis. We first informally define these processes as follows: $X_2$ is a continuous-time process that interpolates $W^k$ in time and it will let us avoid dealing with the discrete-time process $W^k$ directly. $X_1$ is the limiting process of $X_2$ when the step-size goes to zero. Finally, $X_3$ is a process whose law converges to the Gibbs measure $\pi$. 
%

In our approach, we will first relate $X_2$ to its limiting process $X_1$. Since it is more challenging to relate $X_1$ to $x^\star$, we will then relate $X_1$ to $X_3$, and $X_3$ to $\pi$. By following a similar approach to \cite{raginsky17a}, we will finally relate $\pi$ to $f^\star$.  
Formally, we decompose the expected suboptimality in the following manner:
\begin{align}\label{eqn:decompRiskFrac}
\nonumber\mathbb{E}f(W^k)-f^*=& \>\Big(\mathbb{E}f(X_2(k\eta)) - \mathbb{E}f(X_1(k\eta))\Big) + \Big(\mathbb{E}f(X_1(k\eta)) - \mathbb{E}f(X_3(k\eta))\Big) + \Big(\mathbb{E}f(X_3(k\eta)) - \mathbb{E}f(\hat{W})\Big)\\
& + \Big(\mathbb{E}f(\hat{W}) - f^*\Big), 
\end{align}
where $X_i(k\eta)$ with $i=1,2,3$ denotes the state reached by the three stochastic processes at time $k\eta$, and $\hat{W}$ is a random variable drawn from $\pi$. We will now formally define the processes $X_1$, $X_2$, and $X_3$.


The first SDE is the continuous-time limit of the FLA algorithm given in \eqref{eqn:fla} and defined as follows for $t\geq0$:
\begin{align}
\mathrm{d}X_1(t)=b_1(X_1(t-),\alpha)\mathrm{d}t + \beta^{-1/\alpha}\mathrm{d}L^{\alpha}(t),
\label{sde:cont_euler}
\end{align}
where the drift function has the following form:
\begin{align*}
b_1 (x,\alpha) &\triangleq - c_\alpha\nabla f(x).
\end{align*}
The second SDE is a \emph{linearly interpolated} version of the discrete-time process $\{W^k\}_{k\in \mathbb{N}_+}$, defined as follows: 
\begin{align}
\mathrm{d}X_2(t)=b_2(X_2,\alpha)\mathrm{d}t + \beta^{-1/\alpha}\mathrm{d}L^{\alpha}(t), 
\end{align}
where $X_2 \equiv \{X_2(t)\}_{t\geq 0}$ denotes the whole process and the drift function is chosen as follows:
\begin{align*}
b_2 (X_2,\alpha) &\triangleq -c_\alpha\sum_{k=0}^{\infty}\nabla f(X_2(j\eta))\mathbb{I}_{[j\eta,(j+1)\eta[}(t). 
\end{align*}
Here, $\mathbb{I}$ denotes the indicator function, i.e.\ $\mathbb{I}_{A}(x) = 1$ if $x \in A$ and $\mathbb{I}_{A}(x) = 0$ if $x \notin A$. It is easy to verify that $X_2(k\eta) = W^k$ for all $k \in \mathbb{N}_+$ \cite{dalalyan2017theoretical,raginsky17a}.

The last SDE is designed in such a way that its solution has the Gibbs distribution as the invariant distribution and is defined as follows: 
\begin{align}
\label{eqn:FracLangevinGibbs}
\mathrm{d}X_3(t)=b(X_3(t-),\alpha)\mathrm{d}t + \beta^{-1/\alpha}\mathrm{d}L^{\alpha}(t),
\end{align}
where the drift is a $d$-dimensional vector whose $i$-th component, $i=1,\ldots,d$, has the following form: 
\begin{align}
\label{eqn:btrue}
(b(x,\alpha))_i &\triangleq -\frac{\mathcal{D}_{x_i}^{\alpha-2}\Big(\phi(x)\frac{\partial f(x)}{\partial x_i}\Big)\Big)}{\phi(x)}.
\end{align}
Here, $\mathcal{D}_{x_i}$ denotes the Riesz derivative along the direction $x_i$ \cite{ortigueira2014}.
With this definition for the drift, we have the following result for the invariant measure of $X_3$, which is an extension of Theorem~\ref{thm:flmc_org} to general $d$ and $\beta$. 
\begin{lemma}\label{lemma:earlyLemma}
The SDE \eqref{eqn:FracLangevinGibbs} with drift $b$ defined by \eqref{eqn:btrue} admits $\pi$ as an invariant distribution of its solution $(X_3(t))_{t\geq0}$. 
\end{lemma}
The process $\{X_3(t)\}_t$ will play an important role in our analysis, since it will enable us to relate $W^k$ to the Gibbs measure $\pi$, whose samples will be close to the global optimum $x^\star$ with high probability \cite{pavlyukevich2007cooling}.


We now state our assumptions that will imply our main result. 
\begin{assumption}\label{assump:boundedGradAtZero}
There exists a constant $B\geq 0$ such that 
\begin{align*}
c_\alpha\Vert\nabla f(0)\Vert\leq B.
\end{align*} 
\end{assumption}

\begin{assumption}\label{assump:HolderContinuity}
The gradient of $f$ is H\"{o}lder continuous with constants $M>0$, $0\leq\gamma<1$:
\begin{align*}
c_\alpha\Vert\nabla f(x) - \nabla f(y)\Vert\leq M\Vert x-y\Vert^{\gamma},&& \forall x,y\in\mathbb{R}^d.
\end{align*}
\end{assumption}

\begin{assumption}\label{assump:dissipative}
For some $m>0$ and $b\geq 0$, $f$ is $(m,b,\gamma)$-dissipative:
\begin{align*}
c_\alpha\langle x,\nabla f(x)\rangle\geq m\Vert x\Vert^{1+\gamma}-b,&& \forall x\in\mathbb{R}^d.
\end{align*}
\end{assumption}

The assumptions \Cref{assump:boundedGradAtZero}-\Cref{assump:dissipative} are mild and when $\gamma = 1$, they become the standard Lipschitz and dissipativity conditions that are often considered in diffusion-based non-convex optimization algorithms \cite{raginsky17a,xu2018global,erdogdu2018global}. However, due to the choice of the $\alpha$-stable L\'{e}vy motion with $\alpha \in (1,2)$, we need to consider a `fractional' version of those assumptions and exclude the case where $\gamma=1$.

In our analysis, we will make a repeated use of the H\"{o}lder and Minkowski inequalities, which require the following condition to hold:
\begin{assumption}\label{assump:conditionsOnParameters}
There exist positive real numbers $p,q,p_1,q_1$ such that
\begin{gather*}
\frac{1}{p} + \frac{1}{q} = \frac{1}{p_1} + \frac{1}{q_1} = 1, \quad q<\alpha, \quad \gamma p<1, \quad \gamma q_1<1 \quad\text{and}\quad  (q-1)p_1<1.
\end{gather*}
\end{assumption}
Even though this assumption looks rather technical, when combined with \Cref{assump:HolderContinuity} and \Cref{assump:dissipative}, it will in fact impose smoothness constraints on $f$. We will discuss this observation in more detail in Section~\ref{sec:discuss}.

Next, we require an ergodicity condition on $X_3$.
\begin{assumption}\label{assump:ergodicity}
The distribution of $X_3(t)$ \eqref{eqn:FracLangevinGibbs} exponentially converges to its unique invariant distribution of \eqref{eqn:FracLangevinGibbs} in Wasserstein metric, i.e., for any $\lambda\geq 0$ such that $\lambda<\alpha$, there exist constants $C$ and $\lambda_*$ such that
\begin{align*}
\mathcal{W}_{\lambda}(\mu_{3t},\pi)\leq C \beta e^{- \lambda_* t /\beta},
\end{align*}
where $\mu_{3t}$ denotes the probability density of $X_3(t)$.
\end{assumption}
This assumption is also very common in SDE based MCMC algorithms \cite{chen2015convergence,simsekli17a}. Recently, \cite{xie2017ergodicity} has shown that the H\"{o}lder continuity and the fractional-dissipativity assumption with $\gamma=1$ on $b$ would be sufficient for proving geometric ergodicity. We believe that their results can be extended for the case where $\gamma<1$ and we leave it as a future work. In the unadjusted Langevin algorithm ($\alpha=2$), the constant $\lambda_*$ turns out to be the \emph{uniform spectral gap} associated with the Gibbs measure $\pi$ and it has shown to scale exponentially with respect to the dimension $d$ in the worst case \cite{raginsky17a}. We believe that a similar property holds in our case as well.  


Our next assumption is on the approximation quality of the function $b$ by $b_1$.
\begin{assumption}\label{assump:uniformlyBounded}
There exists a constant $L> 0$ such that $L<m$ and 
\begin{align*}
\sup_{x \in \mathbb{R}^d} \Vert c_{\alpha}\nabla f(x)+b(x,\alpha)\Vert\leq L,
\end{align*}
where the function $b$ is defined in \eqref{eqn:btrue}.
\end{assumption}
In Corollary 2 of \cite{simsekli17a}, it has been shown that \Cref{assump:uniformlyBounded} holds if the tails of $\pi$ vanish sufficiently quickly. On the other hand, the gap between $b$ and $b_1$ can be diminished even more if we consider a more sophisticated numerical approximation scheme, such as the one given in \cite{ccelik2012crank} (cf.\ Theorem 2 of \cite{simsekli17a}). 

In our final condition, we assume that the fractional moments of $\pi$ is uniformly bounded. 
\begin{assumption}\label{assump:momentsOfInvariant} There exists a constant $C>0$ such that 
\begin{align*}
\int_{\mathbb{R}^d}  \Vert x\Vert^r \pi(\mathrm{d}x) \leq C\frac{b+d/\beta}{m}
\end{align*}
for all $0\leq r\leq 2$. 
\end{assumption}


Now, we are ready to state our main result. 

\begin{theorem}\label{thm:ultimateTheorem} 
Under conditions \Cref{assump:boundedGradAtZero}-\Cref{assump:momentsOfInvariant} and for $0<\eta<\frac{m}{M^2}$, there exists a positive constant $C$ independent of $k$ and $\eta$ such that the following bound holds:
\begin{align*}
\mathbb{E}[f(W^k)] - f^*\leq& C\Bigg\{k^{1+\max\{\frac{1}{q},\gamma+\frac{\gamma}{q}\}}\eta^{\frac{1}{q}}  +  \frac{k^{1+\max\{\frac{1}{q},\gamma+\frac{\gamma}{q}\}}\eta^{\frac{1}{q}+\frac{\gamma}{\alpha q}}d}{\beta^{\frac{(q-1)\gamma}{\alpha q}}} + \frac{\beta b+d}{m}\exp(- \frac{\lambda_* k\eta} {\beta}) \Bigg\} + \frac{ Mc_\alpha^{-1}}{\beta^{\gamma+1}(1+\gamma)}\\
& + \frac1{\beta}\log\frac{(2e(b+\frac{d}{\beta}))^{\frac{d}{2}}\Gamma(\frac{d}{2}+1)\beta^d}{(dm)^{\frac{d}{2}}}.
\end{align*}
\end{theorem}
More explicit constants can be found in the supplementary document. 
Similar to ULA \cite{raginsky17a}, our bound grows with the number of iterations $k$. We note that this result sheds light on the explicit dependency of the error with respect to the algorithm parameters (e.g.\ step-size) for a fixed number of iterations, rather than explaining the asymptotic behavior when $k$ goes to infinity.  
In the next sections, we will provide an overview of the proof of this theorem along with some remarks and comparisons to ULA. 
%




\section{Proof Overview}
\label{sec:proof_overview}



Our proof strategy consists of bounding each of the four terms in \eqref{eqn:decompRiskFrac} separately. 
Before bounding these terms, we first start by relating the expected discrepancies to the Wasserstein distance between two random processes. The result is formally presented in the following lemma and it extends the 2-Wasserstein continuity result given in \cite{polyanskiy2016wasserstein} to Wasserstein distance with fractional orders. 
%
%
\begin{lemma} \label{lemma:DifferenceExpectation}Let $V$ and $W$ be two random variables on $\mathbb{R}^d$ which have $\mu$ and $\nu$ as the probability measures and let $g$ be a function in $C^1(\mathbb{R}^d,\mathbb{R})$. Assume that for some $c_1>0,c_2\geq 0$ and $0\leq\gamma<1$,
\begin{align*}
\Vert \nabla g(x)\Vert\leq c_1\Vert x\Vert^{\gamma} + c_2,&&\forall x\in\mathbb{R}^d
\end{align*}
and $\max\Big\{\Big(\mathbb{E}\Vert W\Vert^{\gamma p}\Big)^{\frac{1}{p}},\Big(\mathbb{E}\Vert V\Vert^{\gamma p}\Big)^{\frac{1}{p}}\Big\}<\infty$. Then, the following bound holds:
\begin{align*}
\Big\vert\int g\text{d}\mu - \int g\text{d}\nu\Big\vert\leq& C \> \mathcal{W}_{q}(\mu, \nu),
\end{align*}
for some $C>0$.
\end{lemma}
Lemma~\ref{lemma:DifferenceExpectation} lets us upperbound the first three terms of the right hand side of \eqref{eqn:decompRiskFrac} by the Wasserstein distance between the appropriate stochastic processes, respectively $\mathcal{W}_{q}(\mu_{1t}, \mu_{2t})$, $\mathcal{W}_{q}(\mu_{1t}, \mu_{3t})$, and $\mathcal{W}_{q}(\mu_{3t}, \pi)$, where $\mu_{it}$ denotes the law of $X_i(t)$. 

The term $\mathcal{W}_{q}(\mu_{3t}, \pi)$ is related to the ergodicity of the process \eqref{eqn:FracLangevinGibbs}  and it has been shown that this distance diminishes exponentially for a considerably large class of L\'{e}vy diffusions \cite{masuda2007ergodicity,xie2017ergodicity}. On the other hand, the term $\mathcal{W}_{q}(\mu_{1t}, \mu_{3t})$ is related to the numerical approximation of the Riesz derivatives, which is analyzed in \cite{simsekli17a}. Therefore, in this study, we use the assumptions \Cref{assump:ergodicity} and \Cref{assump:uniformlyBounded} for dealing with these terms, and focus on the term $\mathcal{W}_{q}(\mu_{1t}, \mu_{2t})$, which is related to the so-called `weak-error' of the Euler scheme for the SDE \eqref{sde:cont_euler}. The existing estimates for such weak-errors are typically of order $C\eta^a$, where $a<1$ and $C$ is a constant that grows exponentially with $t$ \cite{mikulevivcius2011rate}. The exponential growth with $t$ is prohibitive in our case and one of our main technical contributions is that, in the sequel, we will prove a bound that grows \emph{polynomially} with $t$, which substantially improves over the one with exponential growth.


We start by bounding $\mathcal{W}_{q}(\mu_{1t}, \mu_{2t})$ and $\mathcal{W}_{q}(\mu_{1t}, \mu_{3t})$. In order to do so, we prove the following lemma, which will be the key for our analysis.
\begin{lemma}\label{lemma:WassersteinFormula}
For $\lambda\in(1,\infty)$, $i,j\in\{1,2,3\}$ and $i\neq j$, we have the following identity:
\begin{align*}
\mathcal{W}_{\lambda}(\mu_{it}, \mu_{jt})= \inf\Big\{\Big(\mathbb{E}\Big[\int_0^t \lambda\,\Vert \Delta X_{ij}(s)\Vert^{\lambda -2} \langle \Delta X_{ij}(s) , \Delta b_{ij}(s-)\rangle\text{d}s\Big]\Big)^{1/\lambda}\Big\},
\end{align*}
where the infimum is taken over the couplings whose marginals are $\mu_{it}$ and $\mu_{jt}$ and 
\begin{align*}
\Delta X_{ij}(s) &\triangleq X_i(s) -X_j(s), \quad\Delta b_{ij}(s-) \triangleq b_i (X_i(s-),\alpha)-b_j (X_j(s-),\alpha).
\end{align*}
\end{lemma}
This result extends the recent study \cite{gairing2018transportation} and lets us relate the Wasserstein distance between the distributions of the random processes to their drift functions.

By using Lemma~\ref{lemma:WassersteinFormula}, we start by bounding the Wasserstein distance between $\mu_{1t}$ and $\mu_{2t}$. The result is summarized in the following theorem. 
\begin{theorem}\label{thm:boundOfWassersteinShort}
Assume that the following condition holds: $0<\eta\leq\frac{m}{M^2}$. 
%
Then, we have
\begin{align*}
\mathcal{W}_{q}^{q}(\mu_{1t}, \mu_{2t}) \leq C q \> \mathrm{Poly}(k,\eta,\beta,d),
\end{align*}
for some $C>0$. 
\end{theorem}
The full statement of the proof and the explicit constants are provided in the supplementary document. By only considering the leading terms of the bound provided in Theorem~\ref{thm:boundOfWassersteinShort}, we obtain the following corollary.
\begin{corollary}\label{cor:BoundWasserstein} 
Suppose that $0<\eta<\min\big\{1,\frac{m}{M^2}\big\}$. Then,
 the bound for the Wasserstein distance between the laws of $X_1(t)$ and $X_2(t)$ can be written as follows:
\begin{align*}
\mathcal{W}^q_{q}(\mu_{1t}, \mu_{2t})\leq &C(k^2\eta + k^2\eta^{1+\gamma/\alpha}\beta^{-(q-1)\gamma/\alpha}d).
\end{align*}
\end{corollary}
By combining Corollary~\ref{cor:BoundWasserstein} with Lemma~\ref{lemma:DifferenceExpectation}, we obtain the following result, which provides an upperbound for the first term of the right hand side of \eqref{eqn:decompRiskFrac}.
\begin{corollary}\label{cor:distanceFromX1ToX2_2} For $0<\eta<\frac{m}{M^2}$, there exists a constant $C>0$ such that the following bound holds: 
\begin{align*}
\big\vert\mathbb{E}[f(X_1(k\eta))] - \mathbb{E}[f(X_2(k\eta))]\big\vert \leq C\Big(k^{1+\frac{1}{q}}\eta^{\frac{1}{q}}  +  k^{1+\frac{1}{q}}\eta^{\frac{1}{q}+\frac{\gamma}{\alpha q}}\beta^{-\frac{(q-1)\gamma}{\alpha q}}d\Big).
\end{align*}
\end{corollary}


\begin{remark}
For any $\varepsilon>0$, if we choose $k \simeq \varepsilon^{-1}\mathrm{Poly}(\beta,d)$ and $\eta < \varepsilon^{2q+1}\mathrm{Poly}(\beta,d)$, then the bound in Corollary~\ref{cor:distanceFromX1ToX2_2} scales as $\varepsilon\mathrm{Poly}(\beta,d)$.
\end{remark}




Next, by using a similar approach, we bound the distance between $\mu_{1t}$ and $\mu_{3t}$. In the next theorem, we show that the error grows polynomially with the parameters.
\begin{theorem}\label{thm:WassersteinOfX1X3}
We have the following estimate:
\begin{align*}
\mathcal{W}_{q}^{q}(\mu_{1t}, \mu_{3t})\leq& C q \mathrm{Poly}(k,\eta,\beta,d)
\end{align*}
\end{theorem}
By considering the leading terms of the bound in Theorem~\ref{thm:WassersteinOfX1X3} and combining it with Lemma~\ref{lemma:DifferenceExpectation}, we obtain the following corollaries.
\begin{corollary}\label{cor:WassersteinOfX1X3}
There exists a constant $C\geq0$ such that the following bound holds:
\begin{align*}
\mathcal{W}_{q}^{q}(\mu_{1t}, \mu_{3t})\leq& C(k^{q+\gamma}\eta + k^{q+\gamma}\eta^{q}\beta^{-\frac{q-1}{\alpha}}d)
\end{align*}
\end{corollary}
\begin{corollary}\label{cor:distanceFromX1toX3} There exists a constant $C\geq0$ such that the following inequality holds:
\begin{align*}
\vert\mathbb{E}[f(X_1(k\eta))] - \mathbb{E}[f(X_3(k\eta))]\vert \leq C\Big(k^{\gamma+\frac{\gamma+q}{q}}\eta^{\gamma+\frac{1}{q}}\beta^{-\frac{\gamma}{\alpha}}d + k^{\gamma+\frac{\gamma + q}{q}}\eta^{\frac{1}{q}}\Big).
\end{align*}
\end{corollary}

\begin{remark}
For any $\varepsilon>0$, if we choose $k\simeq \varepsilon^{-1}\mathrm{Poly}(\beta,d)$ and $\eta < \varepsilon^{2q+\gamma q+\gamma}\mathrm{Poly}(\beta,d)$, then the bound in Corollary~\ref{cor:distanceFromX1toX3} scales as $\varepsilon\mathrm{Poly}(\beta,d)$.
\end{remark}



We now pass to the term $\mathbb{E}f(X_3(k\eta)) - \mathbb{E}f(\hat{W})$ of \eqref{eqn:decompRiskFrac}. Since we already assumed that $\mu_{3t}$ exponentially converges to $\pi$ in Wasserstein distance (cf.\ \Cref{assump:ergodicity}), as a direct application of Lemma~\ref{lemma:DifferenceExpectation}, we obtain the following result.
\begin{lemma}\label{lemma:distanceFromX3ToInvariant2} Let $\hat{W}$ be a random variable drawn from the invariant measure $\pi\propto\exp(-\beta f)$ of \eqref{eqn:FracLangevinGibbs}. There exists a constant $C\geq0$ such that the following bound holds:
\begin{align*}
\vert\mathbb{E}[f(X_3(t))] - \mathbb{E}[f(\hat{W})]\vert\leq C\frac{b\beta+d}{m}\exp(- \lambda_*\beta^{-1}t).
\end{align*}
\end{lemma}

\begin{remark}
For any $\varepsilon>0$, if we take $k\eta > \frac{\beta}{\lambda_*} \log\Bigl(\frac1{\varepsilon}\Bigr)$, then the bound in \Cref{lemma:distanceFromX3ToInvariant2} can be scaled as $\varepsilon\mathrm{Poly}(\beta,d)$.
\end{remark}




We finally bound the term $\mathbb{E}f(\hat{W}) - f^*$, which is the expected suboptimality of a sample from $\pi$. By following a similar proof technique presented in \cite{raginsky17a}, we obtain the following result.
\begin{lemma}\label{lem:distanceInvariantToMinimumInExpectation} For $\beta>0$, we have
\begin{align*}
\mathbb{E}[f(\hat{W})] - f^\star \leq& \beta^{-1}\log\Big(\frac{(2e(b+\frac{d}{\beta}))^{d/2}\Gamma(\frac{d}{2}+1)\beta^d}{(dm)^{d/2}}\Big) + \frac{\beta^{-\gamma-1} Mc_\alpha^{-1}}{1+\gamma}.
\end{align*}
\end{lemma}


Combining Corollary~\ref{cor:distanceFromX1ToX2_2}, Corollary~\ref{cor:distanceFromX1toX3}, Lemma~\ref{lemma:distanceFromX3ToInvariant2}, and Lemma~\ref{lem:distanceInvariantToMinimumInExpectation} proves Theorem~\ref{thm:ultimateTheorem}.


\section{Additional Remarks}
\label{sec:discuss}

\subsection{Comparison with ULA}

Let us compare this result with those for ULA presented in \cite{raginsky17a}, since they use a similar decomposition (as opposed to \cite{xu2018global}). The last two terms of the right hand side of the bound in Theorem~\ref{thm:ultimateTheorem} have less importance as they can be made arbitrarily small by increasing $\beta$. Besides, for $\beta$ large enough, the first two terms in our bound can be combined in a single term that scales in the order of $k^{1+\max\{\frac{1}{q},\gamma+\frac{\gamma}{q}\}}\eta^{\frac{1}{q}}$. The corresponding term for ULA is given as follows: $k \eta^{5/4}$, cf.\ Section 3.1 of \cite{raginsky17a}. This observation shows that FLA has a worse dependency both on $k$ and $\eta$, which is not surprising and indeed in-line with the existing literature \cite{mikulevivcius2011rate}. 

\subsection{Discussion on smoothness assumptions}


%
In this section we will discuss Assumption \Cref{assump:conditionsOnParameters} and provide more intuition on its implications. Let us recall the four constraints given in \Cref{assump:conditionsOnParameters}:
\begin{gather*}
(1/{p} + 1/{q}) = (1/{p_1} + 1/{q_1}) = 1,\quad \gamma p<1, \quad  \gamma q_1<1, \quad (q-1)p_1<1.
\end{gather*}
We will refer to these conditions as the \emph{first}, \emph{second}, \emph{third}, and \emph{fourth} conditions, respectively.
Our aim is to find a condition on $\gamma$ (more precisely, the maximum value of $\gamma$) such that there exist $p,q,p_1,q_1>0$ satisfying these four conditions.

First, suppose that $p>q_1$. Then, the maximum value of $\gamma$ is decided by the second constraint. Since we want $\gamma$ to be as large as possible, it is natural to choose a smaller $p$. We can observe that, as we decrease $p$, due to the first and the fourth constraints, the value of $q_1$ needs to be increased. If we continue decreasing $p$, then $q_1$ continues to be increased and soon becomes strictly greater than $p$. At this moment, the maximum value of $\gamma$ is decided by the third constraint, not by the second constraint anymore, and from this point on, it is more plausible to decrease $q_1$. 

By this intuition, it is reasonable to choose $p$ to be equal to $q_1$, which implies that $p_1=q$. Accordingly, the fourth constraint becomes: $(q-1)q<1$.
%
By noting that $q>1$, solving this constraint gives $1<q<(1+\sqrt{5})/2$. Then by the first constraint, we have $p>(3+\sqrt{5})/2$, and the second constraint gives $\gamma<1/p<(3-\sqrt{5})/2$.

This upper bound for $\gamma$ is a number between $0.38$ and $0.39$ and tells us that there exist $p,q,p_1,q_1$ satisfying the four constraints if and only if $0\leq\gamma<(3-\sqrt{5})/2$.

Let us take a closer look at \Cref{thm:ultimateTheorem}. Since $\gamma(q+1)<(3-\sqrt{5})(3+\sqrt{5})/4=1$, we have $\gamma+\gamma/q=\gamma(q+1)/q<1/q$ Hence,
\begin{align*}
1+\max\{1/q,\gamma+\gamma/q\}=1+1/q.
\end{align*}
Let $\varepsilon_1$ and $\varepsilon_2$ be positive numbers such that
\begin{align*}
&1/q-\varepsilon_1 = 2/(1+\sqrt{5})=(\sqrt{5}-1)/2,\\
&\gamma + \varepsilon_2 = (3-\sqrt{5})/2.
\end{align*}
then, if $q=p_1$ is approximately equal to $(1+\sqrt{5})/2$ and $\gamma$ is approximately equal to $(3-\sqrt{5})/2$, we imply that $\varepsilon_1$ and $\varepsilon_2$ become very small and
\begin{align*}
&1/q\approx(\sqrt{5}-1)/2,\\
&1/q+\gamma/(\alpha q)\approx(\sqrt{5}-1)/2+(\sqrt{5}-2)/\alpha,\\
&(q-1)\gamma/(q\alpha)\approx(7-3\sqrt{5})/(2\alpha).
\end{align*}

As a final remark on this smoothness condition, we note that similar constraints are imposed on L\'{e}vy-driven SDEs in other studies as well \cite{panloup2008recursive,simsekli17a}. This is due to the fact that such SDEs often require better-behaved drifts in order to be able to compensate the jumps incurred by the L\'{e}vy motion.



\section{Extensions}
\label{sec:extensions}

\subsection{Guarantees for Posterior Sampling}
\label{sec:sampling}

In this section, we will discuss the implications of our results in the classical Monte Carlo sampling context. If our aim is only to draw samples from the distribution $\pi$, then, for a fixed $k$, we can bound the Wasserstein distance between the law of $W^k$ and $\pi$. The result is stated as follows:
\begin{corollary}\label{thm:WassersteinForSampling}
For $0<\eta\leq\frac{m}{M^2}$, the following bound holds:
\begin{align*}
\mathcal{W}_{q}(\mu_{2t}, \pi)\leq& C\Big(k^{\frac{\max\{2,q+\gamma\}}{q}}\eta^{\frac{1}{q}} + k^{\frac{\max\{2,q+\gamma\}}{q}}\eta^{\frac{1}{q}+\frac{\gamma}{q\alpha}}\beta^{-\frac{\gamma(q-1)}{q\alpha}}d^{\frac{1}{q}} + \beta e^{- \lambda_* \frac{k\eta}{\beta}}\Big).
\end{align*}
\end{corollary}
As a typical use case, we can consider Bayesian posterior sampling, where we choose $\beta =1$ and 
\begin{align*}
f(X) = -(\log \mathrm{P}(Y|X) + \log \mathrm{P}(X)).
\end{align*}
Here, $Y$ denotes a dataset, $\mathrm{P}(Y|X)$ is the likelihood, $\mathrm{P}(X)$ denotes the prior density, and the target distribution $\pi$ becomes the posterior distribution with density $\mathrm{P}(X|Y)$.

\subsection{Extension to Stochastic Gradients}
\label{sec:stoch_grad}
In many machine learning problems, the function $f$ to be minimized has the following form:
\begin{align*}
f(x)\triangleq\frac{1}{n}\sum_{i=1}^{n}f^{(i)}(x),
\end{align*}
where $i$ denotes different data points and $n$ is the total number of data points. In large-scale applications, $n$ can be very large, which renders the gradient computation infeasible. Therefore, at iteration $k$, we often approximate $\nabla f$ by its stochastic version that is defined as follows:
\begin{align*}
\nabla f_k(x)\triangleq\frac{1}{n_s}\sum_{i\in \Omega_k } \nabla f^{(i)}(x),
\end{align*}
where $\Omega_k$ is a random subset of $\{1,\ldots,n\}$ with $\vert \Omega_k\vert=n_s\ll n$. The quantity $\nabla f_k(x)$ is often referred to as the `stochastic gradient'. If the stochastic gradients satisfy a moment condition, then we have the following results:
\begin{theorem}\label{thm:StochGradForX1X2}
Assume that for each $i$, the function $x\mapsto f^{(i)}(x)$ satisfies the conditions \Cref{assump:boundedGradAtZero}-\Cref{assump:momentsOfInvariant}. Let us replace $\nabla f$ by $\nabla f_k$ in \eqref{eqn:fla}. If, in addition, there exists $\delta\in[0,1)$ for any $k$, such that
\begin{align*}
\mathbb{E}_{\Omega_k}\Vert c_\alpha(\nabla f(x) - \nabla f_k(x))\Vert^{q_1}\leq&\delta^{q_1}M^{q_1}\Vert x\Vert^{\gamma q_1}, 
\end{align*}
for $x\in\mathbb{R}^d$, then we have the following bound:
\begin{align*}
\mathcal{W}_{q}^{q}(\mu_{1t}, \mu_{2t})\leq& C(1+\delta)(k^2\eta + k^2\eta^{1+\gamma/\alpha}\beta^{-\gamma(q-1)/\alpha}d).
\end{align*}
\end{theorem}
Similar to our previous bounds, we can use Theorem~\ref{thm:StochGradForX1X2} for obtaining a bound for the expected discrepancy, given as follows:
\begin{corollary}\label{cor:StochGradForX1X2}
Under the same assumptions as in \Cref{thm:StochGradForX1X2}, we have the following bound:
\begin{align*}
\big\vert\mathbb{E}[f(X_1 (k&\eta))] - \mathbb{E}[f(X_2(k\eta))]\big\vert\leq C(1+\delta)\Big(k^{1+\frac{1}{q}}\eta^{\frac{1}{q}}  +  k^{1+\frac{1}{q}}\eta^{\frac{1}{q}+\frac{\gamma}{\alpha q}}\beta^{-\frac{(q-1)\gamma}{\alpha q}}d\Big).
\end{align*}
\end{corollary}
These results show that the guarantees for FLA will still hold even under the presence of stochastic gradients. 

\section{Conclusion}
%
In this study, we focused on FLA, which is a recent extension of ULA, and can be seen as a perturbed version of the gradient descent algorithm with heavy-tailed $\alpha$-stable noise. We analyzed the non-asymptotic behavior of FLMC for non-convex optimization and proved finite-time bounds for its expected suboptimality. Our results agreed with the existing related work, and showed that the weak-error of FLA increases faster than ULA, which suggests using smaller step-sizes in FLA. We finally extended our results to the case where exact gradients are replaced by stochastic gradients and showed that similar results hold in this setting as well. A clear future direction implied by our results is the investigation of the local behavior of FLA.

%
%

\section*{Acknowledgments}

This work is partly supported by the French National Research Agency (ANR) as a part of the FBIMATRIX (ANR-16-CE23-0014) and KAMoulox (ANR-15-CE38-0003-01) projects, and by the industrial chair Machine Learning for Big Data from T\'{e}l\'{e}com ParisTech.

\clearpage
\newpage

\bibliography{./figures/levy,./figures/hamcmc}
\bibliographystyle{unsrt}

\clearpage
\newpage
\section{Appendix}

 \renewcommand{\theequation}{S\arabic{equation}}
 \renewcommand{\thefigure}{S\arabic{figure}}
 \renewcommand{\thelemma}{S\arabic{lemma}}
 \renewcommand{\thecorollary}{S\arabic{corollary}}
 \renewcommand{\thetheorem}{S\arabic{theorem}}
 \renewcommand{\thesection}{S\arabic{section}}

\subsection{Proof of Lemma~\ref{lemma:earlyLemma} }

\begin{proof}
Let $q(X,t)$ be the probability density of $X(t)$. By Proposition 1 in \cite{schertzer2001fractional} (see also Section 7 of the same study), the fractional Fokker-Planck equation associated with \eqref{eqn:FracLangevinGibbs} is given as follows:

\begin{align*}
\partial_tq(X,t)=-\sum_{i=1}^d\frac{\partial[(b(X,\alpha))_i q(X,t)]}{\partial X_i}- \beta^{-1} \sum_{i=1}^d \mathcal{D}_{X_i}^\alpha q(X,t).
\end{align*}
Using definition \eqref{eqn:btrue} of $b$, we have

\begin{align*}
\partial_tq(X,t)
=& -\sum_{i=1}^d \frac{\partial}{\partial_{X_i}}[\frac{\beta^{-1}\mathcal{D}_{X_i}^{\alpha-2}(-\beta\phi(X)\frac{\partial f(X)}{\partial_{X_i}})}{\phi(X)}q(X,t)] - \beta^{-1}\sum_{i=1}^d \mathcal{D}_{X_i}^\alpha q(X,t)\\
=& -\sum_{i=1}^d \frac{\partial}{\partial_{X_i}}[\frac{\beta^{-1}\mathcal{D}_{X_i}^{\alpha-2}(-\beta\pi(X)\frac{\partial f(X)}{\partial_{X_i}})}{\pi(X)}q(X,t)] - \beta^{-1}\sum_{i=1}^d \mathcal{D}_{X_i}^\alpha q(X,t)\\
=& -\sum_{i=1}^d \frac{\partial}{\partial_{X_i}}[\frac{\beta^{-1}\mathcal{D}_{X_i}^{\alpha-2}(\frac{\partial\pi(X)}{\partial_{X_i}})}{\pi(X)}q(X,t)] - \beta^{-1}\sum_{i=1}^d \mathcal{D}_{X_i}^\alpha q(X,t).
\end{align*}
Here, we used $\pi(X)=\phi(X)/\int\phi(X)\text{d}X$ in the second equality and $-\beta\frac{\partial}{\partial X_i} f(X)=\frac{\partial}{\partial X_i}\log\pi(X)=\frac{\partial\pi(X)/\partial X_i}{\pi(X)}$ in the third equality. Next, by replacing $q$ by $\pi$ on the right hand side of the above equality, we have:

\begin{align*}
-\sum_{i=1}^d \frac{\partial}{\partial_{X_i}}[\frac{\beta^{-1}\mathcal{D}_{X_i}^{\alpha-2}(\frac{\partial\pi(X)}{\partial_{X_i}})}{\pi(X)}\pi(X,t)] - \beta^{-1}\sum_{i=1}^d \mathcal{D}_{X_i}^\alpha &\pi(X,t)\\
=& -\sum_{i=1}^d \frac{\partial}{\partial_{X_i}}[\beta^{-1}\mathcal{D}_{X_i}^{\alpha-2}(\frac{\partial\pi(X)}{\partial_{X_i}})] - \beta^{-1}\sum_{i=1}^d \mathcal{D}_{X_i}^\alpha \pi(X,t)\\
=&-\sum_{i=1}^d \frac{\partial^2}{\partial_{X_i^2}}[\beta^{-1}\mathcal{D}_{X_i}^{\alpha-2}(\pi(X))] - \beta^{-1}\sum_{i=1}^d \mathcal{D}_{X_i}^\alpha \pi(X,t)\\
=& \sum_{i=1}^d \mathcal{D}_{X_i}^2[\beta^{-1}\mathcal{D}_{X_i}^{\alpha-2}(\pi(X))] - \beta^{-1}\sum_{i=1}^d \mathcal{D}_{X_i}^\alpha \pi(X,t)\\
=& \sum_{i=1}^d \mathcal{D}_{X_i}^{\alpha}[\beta^{-1}\pi(X)] - \beta^{-1}\sum_{i=1}^d \mathcal{D}_{X_i}^\alpha \pi(X,t) = 0.
\end{align*}
Here, we used Proposition 1 in \cite{csimcsekli2017fractional}, $\mathcal{D}^2u(x)=-\frac{\partial}{\partial x^2}u(x)$, and the semi-group property of the Riesz derivation $\mathcal{D}^a\mathcal{D}^bu(x)=\mathcal{D}^{a+b}u(x)$. This proves that $\pi$ is an invariant measure of the Markov process $(X(t))_{t\geq0}$. 
\end{proof}

\subsection{Proof of Lemma~\ref{lemma:DifferenceExpectation}}
In this section, we precise the statement of Lemma~\ref{lemma:DifferenceExpectation} and provide the proof.

\begin{lemma} Let $V$ and $W$ be two random variables on $\mathbb{R}^d$ which have $\mu$ and $\nu$ as the probability measures and let $g$ be a function in $C^1(\mathbb{R}^d,\mathbb{R})$. Assume that for some $c_1>0,c_2\geq 0$ and $0\leq\gamma<1$,
\begin{align*}
\Vert \nabla g(w)\Vert\leq c_1\Vert w\Vert^{\gamma} + c_2,&&\forall w\in\mathbb{R}^d
\end{align*}
then the following bound holds:
\begin{align*}
\Big\vert\int g\text{d}\mu - \int g\text{d}\nu\Big\vert\leq \Big(c_1\Big(\mathbb{E}_{\mathbf{P}}\Vert W\Vert^{\gamma p}\Big)^{\frac{1}{p}}+c_1\Big(\mathbb{E}_{\mathbf{P}}\Vert V\Vert^{\gamma p}\Big)^{\frac{1}{p}}+c_2\Big)\mathcal{W}_{q}(\mu, \nu).
\end{align*}
\end{lemma}
\begin{proof} We have
\begin{align*}
g(v)-g(w)&=\int_0^1\langle w-v,\nabla g((1-t)v+tw)\rangle\text{d}t\\
&\leq\int_0^1\Vert w-v\Vert\Vert\nabla g((1-t)v+tw)\Vert\text{d}t &&\text{(by Cauchy-Schwarz)}\\
&\leq\int_0^1\Vert w-v\Vert( c_1((1-t)\Vert v\Vert+ t\Vert w\Vert)^{\gamma}+c_2)\text{d}t &&\text{(by the assumption on $\nabla g$)}\\
&\leq\Vert w-v\Vert\Big(c_1( \Vert v\Vert+ \Vert w\Vert)^{\gamma}+c_2\Big)\\
&\leq\Vert w-v\Vert( c_1\Vert v\Vert^{\gamma}+c_1 \Vert w\Vert^{\gamma}+c_2). &&\text{(by lemma \ref{lemma:anUsefulIneq})}
\end{align*}
Now let $\mathbf{P}$ be a joint probability distribution of $\mu$ and $\nu$ that achieves $\mathcal{W}_{\lambda}(\mu, \nu)$, that is, $\mathbf{P}=\mathcal{L}((W,V))$ with $\mu=\mathcal{L}(W)$ and $\nu=\mathcal{L}(V)$. We have
\begin{align*}
\int g\text{d}\mu - \int g\text{d}\nu &=\mathbb{E}_{\mathbf{P}} [g(W)-g(V)]\\
&\leq[\mathbb{E}_{\mathbf{P}}(c_1\Vert W\Vert^{\gamma} + c_1\Vert V\Vert^{\gamma} + c_2)^{p}]^{\frac{1}{p}}[\mathbb{E}_{\mathbf{P}}\Vert W-V\Vert^{q}]^{\frac{1}{q}}\\
&\leq \Big(c_1\Big(\mathbb{E}_{\mathbf{P}}\Vert W\Vert^{\gamma p}\Big)^{\frac{1}{p}}+c_1\Big(\mathbb{E}_{\mathbf{P}}\Vert V\Vert^{\gamma p}\Big)^{\frac{1}{p}}+c_2\Big)\mathcal{W}_{q}(\mu, \nu),
\end{align*}
where we have used Holder's inequality and Minkowski's inequality.
\end{proof}

\subsection{Proof of Lemma~\ref{lemma:WassersteinFormula}}

\begin{proof}
We define a real function $F_{\lambda}$ as follows:

\begin{align}\label{eqn:pEuclideDistance}
F_{\lambda}(y)\triangleq\Vert y\Vert^{\lambda}.
\end{align}
It is clear that $F_{\lambda}$ is a $C^1$ function. Let $Y(t)\triangleq X_1(t)-X_2(t)$. By the chain rule,

\begin{align}\label{eqn:ItoFormula}
\nonumber\text{d}F_{\lambda}(Y(t))&=\langle\nabla F_{\lambda}(Y(t)),b_1 (X_1(t-),\alpha)-b_2 (X_2(t-),\alpha)\rangle\text{d}t\\
&=\lambda\,\Vert X_1(t)-X_2(t))\Vert^{\lambda -2} \langle X_1(t)-X_2(t),b_1 (X_1(t-),\alpha)-b_2 (X_2(t-),\alpha)\rangle\text{d}t.
\end{align}
By integrating both sides of \eqref{eqn:ItoFormula} with respect to $t$, we arrive at
\begin{align*}
F_{\lambda}(Y(t))&=F_{\lambda}(Y(0)) + \int_0^t \lambda\,\Vert X_1(t)-X_2(t))\Vert^{\lambda -2} \langle X_1(t)-X_2(t),b_1 (X_1(t-),\alpha)-b_2 (X_2(t-),\alpha)\rangle\text{d}s\\
&=\int_0^t \lambda\,\Vert X_1(t)-X_2(t))\Vert^{\lambda -2} \langle X_1(t)-X_2(t),b_1 (X_1(t-),\alpha)-b_2 (X_2(t-),\alpha)\rangle\text{d}s.
\end{align*}
By definition of Wasserstein distance, we have

\begin{align*}
\mathcal{W}_{\lambda}(\mu_{1t}, \mu_{2t})=\inf\{(\mathbb{E}[F_{\lambda}(Y(t))])^{1/\lambda}\},
\end{align*}
which is the desired result.
\end{proof}

\subsection{Proof of Theorem~\ref{thm:boundOfWassersteinShort}}

In this section, we first precise the statement of Theorem~\ref{thm:boundOfWassersteinShort} and then provide the corresponding proof.

\begin{theorem}\label{thm:boundOfWasserstein}
Let $\mathbb{E}\Vert L^{\alpha}(1)\Vert^{\lambda}\triangleq l_{\alpha,\lambda,d}<\infty$. We also define the following quantities:
\begin{align*}
& P_1(\eta)\triangleq \Big(c\eta\Big(\frac{d}{\beta^{1/\alpha}}\Big)\Big)^{\frac{1}{p_1}} + (c\eta)^{\frac{1}{p_1}} + (2\eta (b+m))^{\frac{(q-1)}{2}} + 2^{\frac{(q-1)}{2}}(\eta B)^{(q-1)} +\Big(\frac{\eta}{\beta}\Big)^{\frac{(q-1)}{\alpha}}l_{\alpha,(q-1)p_1,d}^{\frac{1}{p_1}}\\
 + \eta^{q-1} M^{q-1} \Big((2\eta (b+m))^{\frac{(q-1)\gamma}{2}} + 2^{\frac{(q-1)\gamma}{2}}(\eta B)^{(q-1)\gamma} +\Big(\frac{\eta}{\beta}\Big)^{\frac{(q-1)\gamma}{\alpha}}l_{\alpha,(q-1)p_1\gamma,d}^{\frac{1}{p_1}}\Big), \hspace{-378pt} \\ 
& P_2(\eta)\triangleq M\Big(\Big(c\eta\Big(\frac{d}{\beta^{1/\alpha}}\Big)\Big)^{\frac{1}{q_1}} + (c\eta)^{\frac{1}{q_1}} + (2\eta (b+m))^{\frac{\gamma}{2}} +2^{\frac{\gamma}{2}}(\eta B)^{\gamma} +\Big(\frac{\eta}{\beta}\Big)^{\frac{\gamma}{\alpha}}l_{\alpha,\gamma q_1,d}^{\frac{1}{q_1}}\Big),\\
& Q_1(\eta)\triangleq c^{\frac{1}{p_1}} + (\mathbb{E}\Vert X_2(0)\Vert^{(q-1)p_1})^{\frac{1}{p_1}} + \eta^{q-1}\Big(M^{q-1}(\mathbb{E}\Vert X_2(0)\Vert^{(q-1)p_1\gamma})^{\frac{1}{p_1}} + B^{(q-1)} \Big) + \Big(\frac{\eta}{\beta}\Big)^{\frac{q-1}{\alpha}}l_{\alpha,(q-1)p_1,d}^{\frac{1}{p_1}},\\
& Q_2\triangleq M(\mathbb{E}\Vert X_2(0)\Vert^{\gamma q_1})^{\frac{1}{q_1}} + Mc^{\frac{1}{q_1}}.
\end{align*}
Under additional assumption on the step-size: 
$0<\eta\leq\frac{m}{M^2}$,
we have
\begin{align*}
\mathcal{W}_{q}^{q}(\mu_{1t}, \mu_{2t})\leq q\eta\Big(k^2P_1(\eta)P_2(\eta) + k^{1+1/p_1}P_1(\eta)Q_2 + k^{1+1/q_1}P_2(\eta)Q_1(\eta) + kQ_1(\eta)Q_2\Big).
\end{align*}
\end{theorem}
\begin{proof} 
From Lemma~\ref{lemma:WassersteinFormula}, we have
\begin{align*}
\mathcal{W}_{q}^{q}(\mu_{1t}, \mu_{2t})=&\mathbb{E}\Big[\int_0^t q\,\Vert X_1(s)-X_2(s))\Vert^{q -2} \langle X_1(s)-X_2(s),b_1 (X_1(s-),\alpha)-b_2 (X_2(s-),\alpha)\rangle\text{d}s\Big]\\
=&\sum_{j=0}^{k-1}\mathbb{E}\Big[\int_{j\eta}^{(j+1)\eta}q\,\Vert X_1(s)-X_2(s))\Vert^{q -2} \langle X_1(s)-X_2(s),b_1 (X_1(s-),\alpha)-b_2 (X_2(s-),\alpha)\rangle\text{d}s\Big]\\
\leq&\sum_{j=0}^{k-1}\mathbb{E}\Big[\int_{j\eta}^{(j+1)\eta}q\,\Vert X_1(s)-X_2(s)\Vert^{q-1}c_\alpha\Vert\nabla f(X_1(s))-\nabla f(X_2(j\eta))\Vert\text{d}s\Big]\\
=&q\,\sum_{j=0}^{k-1}\int_{j\eta}^{(j+1)\eta}\mathbb{E}\Big[\Vert X_1(s)-X_2(s)\Vert^{q-1}c_\alpha\Vert\nabla f(X_1(s))-\nabla f(X_2(j\eta))\Vert\Big]\text{d}s\\
\leq& q\,\sum_{j=0}^{k-1}\int_{j\eta}^{(j+1)\eta}\Big[\mathbb{E}\Vert X_1(s)-X_2(s)\Vert^{(q-1)p_1}\Big]^{\frac{1}{p_1}}\Big[\mathbb{E}\Vert c_\alpha(\nabla f(X_1(s))-\nabla f(X_2(j\eta)))\Vert^{q_1}\Big]^{\frac{1}{q_1}}\text{d}s,
\end{align*}
where we have used Cauchy-Schwarz inequality in the third line and Holder's inequality in the last line.

Since $(q-1)p_1< 1$ by Assumption~\Cref{assump:conditionsOnParameters}, using Lemma \ref{lemma:anUsefulIneq} twice, we have:
\begin{align*}
\Big(\mathbb{E}\Vert X_1(s)-X_2(s)\Vert^{(q-1)p_1}\Big)^{\frac{1}{p_1}}\leq&\Big(\mathbb{E}\Vert X_1(s)\Vert^{(q-1)p_1}+\mathbb{E}\Vert X_2(s)\Vert^{(q-1)p_1}\Big)^{\frac{1}{p_1}}\\
\leq&\Big[\mathbb{E}\Big(\Vert X_1(s)\Vert^{(q-1)p_1}\Big)\Big]^{\frac{1}{p_1}}+\Big[\mathbb{E}\Big(\Vert X_2(s)\Vert^{(q-1)p_1}\Big)\Big]^{\frac{1}{p_1}}
\end{align*}
Then, by applying Lemma~\ref{lemma:expectationBoundOfX1} and Lemma~\ref{lemma:expecataionBound} for $s\in[j\eta,(j+1)\eta)$, we obtain:
\begin{align*}
\Big(\mathbb{E}\Vert X_1(s)- & X_2(s)\Vert^{(q-1)p_1}\Big)^{\frac{1}{p_1}}\\
\leq&  \Big(c\Big(s\Big(\frac{d}{\beta^{1/\alpha}}+1\Big)+1\Big)\Big)^{q-1} + \Big[\mathbb{E}\Vert X_2(0)\Vert^{(q-1)p_1} + j\Big((2\eta (b+m))^{\frac{(q-1)p_1}{2}} +2^{\frac{(q-1)p_1}{2}}(\eta B)^{(q-1)p_1}\\
 &+\Big(\frac{\eta}{\beta}\Big)^{\frac{(q-1)p_1}{\alpha}}l_{\alpha,(q-1)p_1,d}\Big) + (s-j\eta)^{(q-1)p_1}\Big(M^{(q-1)p_1} \Big(\mathbb{E}\Vert X_2(0)\Vert^{(q-1)p_1\gamma} + j\Big((2\eta (b+m))^{\frac{(q-1)p_1\gamma}{2}}\\
 &+ 2^{\frac{(q-1)p_1\gamma}{2}}(\eta B)^{(q-1)p_1\gamma} +\Big(\frac{\eta}{\beta}\Big)^{\frac{(q-1)p_1\gamma}{\alpha}}l_{\alpha,(q-1)p_1\gamma,d}\Big)\Big) + B^{(q-1)p_1}\Big) + \Big(\frac{s-j\eta}{\beta}\Big)^{\frac{(q-1)p_1}{\alpha}}l_{\alpha,(q-1)p_1,d}\Big]^{\frac{1}{p_1}}.
\end{align*}
Next, using Lemma~\ref{lemma:anUsefulIneq}, the inequalities $j<j+1$ and $s-j\eta\leq\eta$ for $s\in[j\eta,(j+1)\eta)$, we get
\begin{align*}
\Big(\mathbb{E}\Vert X_1(s)-X_2(s)\Vert^{(q-1)p_1}&\Big)^{\frac{1}{p_1}}\\
\leq&  \Big(c\Big(s\Big(\frac{d}{\beta^{1/\alpha}}+1\Big)+1\Big)\Big)^{q-1} + (\mathbb{E}\Vert X_2(0)\Vert^{(q-1)p_1})^{\frac{1}{p_1}} + (j+1)^{\frac{1}{p_1}}\Big((2\eta (b+m))^{\frac{(q-1)}{2}}\\
& +2^{\frac{(q-1)}{2}}(\eta B)^{(q-1)} +\Big(\frac{\eta}{\beta}\Big)^{\frac{(q-1)}{\alpha}}l_{\alpha,(q-1)p_1,d}^{\frac{1}{p_1}}\Big) + \eta^{q-1}\Big(M^{q-1} \Big((\mathbb{E}\Vert X_2(0)\Vert^{(q-1)p_1\gamma})^{\frac{1}{p_1}}\\
& + (j+1)^{\frac{1}{p_1}}\Big((2\eta (b+m))^{\frac{(q-1)\gamma}{2}} +2^{\frac{(q-1)\gamma}{2}}(\eta B)^{(q-1)\gamma} +\Big(\frac{\eta}{\beta}\Big)^{\frac{(q-1)\gamma}{\alpha}}l_{\alpha,(q-1)p_1\gamma,d}^{\frac{1}{p_1}}\Big)\Big)\\
& + B^{(q-1)}\Big) + \Big(\frac{\eta}{\beta}\Big)^{\frac{q-1}{\alpha}}l_{\alpha,(q-1)p_1,d}^{\frac{1}{p_1}}.
\end{align*}
We note that $s<(j+1)\eta$ and $q-1<\frac{1}{p_1}$ (from the assumptions). Hence,
\begin{align*}
\Big(c\Big(s\Big(\frac{d}{\beta^{1/\alpha}}+1\Big)+1\Big)\Big)^{q-1}\leq&\Big(c\Big((j+1)\eta\Big(\frac{d}{\beta^{1/\alpha}}+1\Big) + 1\Big)\Big)^{\frac{1}{p_1}}\\
\leq&(j+1)^{\frac{1}{p_1}}\Big(c\eta\Big(\frac{d}{\beta^{1/\alpha}}+1\Big)\Big)^{\frac{1}{p_1}} + c^{\frac{1}{p_1}},
\end{align*}
where the last inequality is an application of Lemma~\ref{lemma:anUsefulIneq}. By replacing this inequality into the previous one and rearranging the terms, we have
\begin{align*}
\Big(\mathbb{E}\Vert X_1(s)&-X_2(s)\Vert^{(q-1)p_1} \Big)^{\frac{1}{p_1}}\\
\leq& c^{\frac{1}{p_1}} + (\mathbb{E}\Vert X_2(0)\Vert^{(q-1)p_1})^{\frac{1}{p_1}} + \eta^{q-1}\Big(M^{q-1}(\mathbb{E}\Vert X_2(0)\Vert^{(q-1)p_1\gamma})^{\frac{1}{p_1}} + B^{(q-1)} \Big) + \Big(\frac{\eta}{\beta}\Big)^{\frac{q-1}{\alpha}}l_{\alpha,(q-1)p_1,d}^{\frac{1}{p_1}}\\
& + (j+1)^{\frac{1}{p_1}}\Big(\Big(c\eta\Big(\frac{d}{\beta^{1/\alpha}}+1\Big)\Big)^{\frac{1}{p_1}} + (2\eta (b+m))^{\frac{(q-1)}{2}} + 2^{\frac{(q-1)}{2}}(\eta B)^{(q-1)} +\Big(\frac{\eta}{\beta}\Big)^{\frac{(q-1)}{\alpha}}l_{\alpha,(q-1)p_1,d}^{\frac{1}{p_1}}\\
&+ \eta^{q-1} M^{q-1} \Big((2\eta (b+m))^{\frac{(q-1)\gamma}{2}} + 2^{\frac{(q-1)\gamma}{2}}(\eta B)^{(q-1)\gamma} +\Big(\frac{\eta}{\beta}\Big)^{\frac{(q-1)\gamma}{\alpha}}l_{\alpha,(q-1)p_1\gamma,d}^{\frac{1}{p_1}}\Big)\Big)\\
\leq& c^{\frac{1}{p_1}} + (\mathbb{E}\Vert X_2(0)\Vert^{(q-1)p_1})^{\frac{1}{p_1}} + \eta^{q-1}\Big(M^{q-1}(\mathbb{E}\Vert X_2(0)\Vert^{(q-1)p_1\gamma})^{\frac{1}{p_1}} + B^{(q-1)} \Big) + \Big(\frac{\eta}{\beta}\Big)^{\frac{q-1}{\alpha}}l_{\alpha,(q-1)p_1,d}^{\frac{1}{p_1}}\\
& + (j+1)^{\frac{1}{p_1}}\Big(\Big(c\eta\Big(\frac{d}{\beta^{1/\alpha}}\Big)\Big)^{\frac{1}{p_1}} + (c\eta)^{\frac{1}{p_1}} + (2\eta (b+m))^{\frac{(q-1)}{2}} + 2^{\frac{(q-1)}{2}}(\eta B)^{(q-1)} +\Big(\frac{\eta}{\beta}\Big)^{\frac{(q-1)}{\alpha}}l_{\alpha,(q-1)p_1,d}^{\frac{1}{p_1}}\\
&+ \eta^{q-1} M^{q-1} \Big((2\eta (b+m))^{\frac{(q-1)\gamma}{2}} + 2^{\frac{(q-1)\gamma}{2}}(\eta B)^{(q-1)\gamma} +\Big(\frac{\eta}{\beta}\Big)^{\frac{(q-1)\gamma}{\alpha}}l_{\alpha,(q-1)p_1\gamma,d}^{\frac{1}{p_1}}\Big)\Big)\\
=&Q_1(\eta) + (j+1)^{\frac{1}{p_1}}P_1(\eta),
\end{align*}
Here, we have used \Cref{lemma:anUsefulIneq} in the last inequality. Now, consider the following quantity
\begin{align*}
\Big[\mathbb{E}\Vert c_\alpha(\nabla f(X_1(s))-\nabla f(X_2(j\eta)))\Vert^{q_1}\Big]^{\frac{1}{ q_1}}\leq&\Big[\mathbb{E}\Big(M\Vert X_1(s)- X_2(j\eta)\Vert^{\gamma}\Big)^{q_1}\Big]^{\frac{1}{q_1}}\\
\leq&\Big[\mathbb{E}\Big(M\Vert X_1(s)\Vert^{\gamma}+M\Vert X_2(j\eta)\Vert^{\gamma}\Big)^{q_1}\Big]^{\frac{1}{q_1}}\\
\leq&\Big[\mathbb{E}\Big(M^{q_1}\Vert X_1(s)\Vert^{\gamma q_1}\Big)\Big]^{\frac{1}{q_1}}+\Big[\mathbb{E}\Big(M^{q_1}\Vert X_2(j\eta)\Vert^{\gamma q_1}\Big)\Big]^{\frac{1}{q_1}},
\end{align*}
where we have used Assumption~\Cref{assump:HolderContinuity}, Lemma~\ref{lemma:anUsefulIneq} and Minkowski's inequality. By Lemma~\ref{lemma:expectationBoundOfX1} and Lemma~\ref{lemma:expecataionBound}, we have
\begin{align*}
\Big[\mathbb{E}\Vert c_\alpha\nabla f(X_1(s))-c_\alpha\nabla f(X_2(j\eta))\Vert^{q_1}\Big]^{\frac{1}{ q_1}}\leq& M\Big(c\Big(s\Big(\frac{d}{\beta^{1/\alpha}}+1\Big)+1\Big)\Big)^{\gamma} + \Big[M^{q_1}(\mathbb{E}\Vert X_2(0)\Vert^{\gamma q_1})\\
& + M^{q_1}j\Big((2\eta (b+m))^{\frac{\gamma q_1}{2}} +2^{\frac{\gamma q_1}{2}}(\eta B)^{\gamma q_1} +\Big(\frac{\eta}{\beta}\Big)^{\frac{\gamma q_1}{\alpha}}l_{\alpha,\gamma q_1,d}\Big)\Big]^{\frac{1}{q_1}}.
\end{align*}
By using Lemma~\ref{lemma:anUsefulIneq} and the inequality $j<j+1$, we have
\begin{align*}
\Big[\mathbb{E}\Vert c_\alpha\nabla f(X_1(s))-c_\alpha\nabla f(X_2(j\eta))\Vert^{q_1}\Big]^{\frac{1}{ q_1}}\leq& M\Big(c\Big(s\Big(\frac{d}{\beta^{1/\alpha}}+1\Big)+1\Big)\Big)^{\gamma} + M(\mathbb{E}\Vert X_2(0)\Vert^{\gamma q_1})^{\frac{1}{q_1}}\\
& + M(j+1)^{\frac{1}{q_1}}\Big((2\eta (b+m))^{\frac{\gamma }{2}} +2^{\frac{\gamma }{2}}(\eta B)^{\gamma } +\Big(\frac{\eta}{\beta}\Big)^{\frac{\gamma }{\alpha}}l_{\alpha,\gamma q_1,d}^{\frac{1}{q_1}}\Big).
\end{align*}
We note that $s<(j+1)\eta$ and $\gamma<\frac{1}{q_1}$ (from the assumptions). Hence,
\begin{align*}
\Big(c\Big(s\Big(\frac{d}{\beta^{1/\alpha}}+1\Big)+1\Big)\Big)^{\gamma}\leq&\Big(c\Big((j+1)\eta\Big(\frac{d}{\beta^{1/\alpha}}+1\Big) + 1\Big)\Big)^{\frac{1}{q_1}}\\
\leq&(j+1)^{\frac{1}{q_1}}\Big(c\eta\Big(\frac{d}{\beta^{1/\alpha}}+1\Big)\Big)^{\frac{1}{q_1}} + c^{\frac{1}{q_1}},
\end{align*}
where the last inequality is an application of Lemma~\ref{lemma:anUsefulIneq}. By replacing this inequality into the previous one and rearranging the terms, we have

\begin{align*}
\Big[\mathbb{E}\Vert c_\alpha\nabla f(X_1(s))-c_\alpha\nabla f(X_2(j\eta))\Vert^{q_1}\Big]^{\frac{1}{ q_1}}
\leq& M(\mathbb{E}\Vert X_2(0)\Vert^{\gamma q_1})^{\frac{1}{q_1}} + Mc^{\frac{1}{q_1}} + M(j+1)^{\frac{1}{q_1}}\Big(\Big(c\eta\Big(\frac{d}{\beta^{1/\alpha}}+1\Big)\Big)^{\frac{1}{q_1}}\\
& + (2\eta (b+m))^{\frac{\gamma}{2}} +2^{\frac{\gamma}{2}}(\eta B)^{\gamma} +\Big(\frac{\eta}{\beta}\Big)^{\frac{\gamma}{\alpha}}l_{\alpha,\gamma q_1,d}^{\frac{1}{q_1}}\Big)\\
\leq& M(\mathbb{E}\Vert X_2(0)\Vert^{\gamma q_1})^{\frac{1}{q_1}} + Mc^{\frac{1}{q_1}} + M(j+1)^{\frac{1}{q_1}}\Big(\Big(c\eta\Big(\frac{d}{\beta^{1/\alpha}}\Big)\Big)^{\frac{1}{q_1}}\\
& + (c\eta)^{\frac{1}{q_1}} + (2\eta (b+m))^{\frac{\gamma}{2}} +2^{\frac{\gamma}{2}}(\eta B)^{\gamma} +\Big(\frac{\eta}{\beta}\Big)^{\frac{\gamma}{\alpha}}l_{\alpha,\gamma q_1,d}^{\frac{1}{q_1}}\Big)\\
=&Q_2 + (j+1)^{\frac{1}{q_1}}P_2(\eta).
\end{align*}
Here, we have used \Cref{lemma:anUsefulIneq} in the last inequality. By combining the above inequalities, we get
\begin{align*}
\mathbb{E}\Big[\int_0^t q\,\Vert X_1(s)-& X_2(s))\Vert^{q -2} \langle X_1(s)-X_2(s),b_1 (X_1(s-),\alpha)-b_2 (X_2(s-),\alpha)\rangle\text{d}s\Big]  \\
&\leq\sum_{j=0}^{k-1}q\eta\Big((j+1)P_1(\eta)P_2(\eta) + (j+1)^{\frac{1}{p_1}}P_1(\eta)Q_2+(j+1)^{\frac{1}{q_1}}P_2(\eta)Q_1(\eta)+Q_1(\eta)Q_2\Big)\\
& \leq q\eta\Big(k^2P_1(\eta)P_2(\eta) + k^{1+1/p_1}P_1(\eta)Q_2 + k^{1+1/q_1}P_2(\eta)Q_1(\eta) + kQ_1(\eta)Q_2\Big).
\end{align*}
The final conclusion follows from this inequality.
\end{proof}

\subsubsection{Proof of Corollary~\ref{cor:BoundWasserstein}}

\begin{proof}
In order to get the results from the bound obtained by Theorem~\ref{thm:boundOfWasserstein}, we take the max power of $k$ and the min power of $\eta$ among the terms containing $k$ and $\eta$ but not containing $\beta$. For the terms containing $\beta$, we take the max power of $k$, min power of $\eta$, min power of $1/\beta$ and max power of $d$. We get
\begin{align*}
\mathcal{W}^q_{q}(\mu_{1t}, \mu_{2t})\leq &C(k^2\eta + k^2\eta^{1+\min\{\gamma,q-1\}/\alpha}\beta^{-(q-1)\gamma/\alpha}d).
\end{align*}
Since $\gamma<1/p=(q-1)/q<q-1$, we finally obtain
\begin{align*}
\mathcal{W}^q_{q}(\mu_{1t}, \mu_{2t})\leq &C(k^2\eta + k^2\eta^{1+\gamma/\alpha}\beta^{-(q-1)\gamma/\alpha}d).
\end{align*}
\end{proof}

\subsubsection{Proof of Corollary~\ref{cor:distanceFromX1ToX2_2}}

\begin{proof}
The proof starts from the bound established in \Cref{cor:distanceFromX1toX2} then, follows the same lines of the proof of Corollary~\ref{cor:BoundWasserstein}. 
\end{proof}

\subsection{Proof of Theorem~\ref{thm:ultimateTheorem}}

\begin{proof}
We have the decomposition:
\begin{align*}
\mathbb{E}[f(W^k)] -& f^* \\
=& \mathbb{E}[f(X_2(k\eta))] - f^* \\
=& (\mathbb{E}[f(X_2(k\eta))] - \mathbb{E}[f(X_1(k\eta))]) + (\mathbb{E}[f(X_1(k\eta))] - \mathbb{E}[f(X_3(k\eta))]) + (\mathbb{E}[f(X_3(k\eta))] - \mathbb{E}[f(\hat{W}))]) \\
&+ (\mathbb{E}[f(\hat{W}))] - f^*).
\end{align*}
By Corollary~\ref{cor:distanceFromX1ToX2_2}, Corollary~\ref{cor:distanceFromX1toX3}, Lemma~\ref{lemma:distanceFromX3ToInvariant2} and Lemma~\ref{lem:distanceInvariantToMinimumInExpectation}, there exists a constant $C'$ independent of $k$, $\eta$ and $\beta$ such that
\begin{align*}
\mathbb{E}[f(W^k)] - f^*\leq &C'\Big(k^{1+\frac{1}{q}}\eta^{\frac{1}{q}}  +  k^{1+\frac{1}{q}}\eta^{\frac{1}{q}+\frac{\gamma}{\alpha q}}\beta^{-\frac{(q-1)\gamma}{\alpha q}}d + k^{\gamma+\frac{\gamma+q}{q}}\eta^{\gamma+\frac{1}{q}}\beta^{-\frac{\gamma}{\alpha}}d + k^{\gamma+\frac{\gamma + q}{q}}\eta^{\frac{1}{q}}\\
& + \beta\frac{b+d/\beta}{m}\exp(- \lambda_*\beta^{-1}t) \Big) + \frac{\beta^{-\gamma-1} Mc_\alpha^{-1}}{1+\gamma} + \beta^{-1}\log\Big(\frac{(2e(b+d/\beta))^{d/2}\Gamma(d/2+1)\beta^d}{(dm)^{d/2}}\Big).
\end{align*}
Here, we note that $k\eta=t$. then by taking the largest power of $k$, smallest powers of $\eta$ and $\beta^{-1}$ among the terms containing all of three parameters $k$, $\eta$ and $\beta$, there exist a constant $C$ satisfying the following inequality:
\begin{align*}
\mathbb{E}[f(W^k)] - f^*\leq &C\Big(k^{1+\max\{\frac{1}{q},\gamma+\frac{\gamma}{q}\}}\eta^{\frac{1}{q}}  +  k^{1+\max\{\frac{1}{q},\gamma+\frac{\gamma}{q}\}}\eta^{\frac{1}{q}+\frac{\gamma}{\alpha q}}\beta^{-\frac{(q-1)\gamma}{\alpha q}}d + \beta\frac{b+d/\beta}{m}\exp(- \lambda_*\beta^{-1}k\eta) \Big)\\
& + \frac{\beta^{-\gamma-1} Mc_\alpha^{-1}}{1+\gamma} + \beta^{-1}\log\Big(\frac{(2e(b+d/\beta))^{d/2}\Gamma(d/2+1)\beta^d}{(dm)^{d/2}}\Big).
\end{align*}
\end{proof}

\subsection{Proof of Theorem~\ref{thm:WassersteinOfX1X3}}
In this section, we precise the statement of Theorem~\ref{thm:WassersteinOfX1X3} and provide the full proof. 
\begin{theorem}
We have the following estimate:
\begin{align*}
\mathcal{W}_{q}^{q}(\mu_{1t}, \mu_{3t})\leq& qt\Big(M(c^{q-1}+c_b^{q-1})(c^{\gamma}+c_b^{\gamma})\Big(t\Big(\frac{d}{\beta^{1/\alpha}}+1\Big)+1\Big)^{q-1+\gamma} + L(c^{q-1}+c_b^{q-1})\Big(t\Big(\frac{d}{\beta^{1/\alpha}}+1\Big)+1\Big)^{q-1}\Big),
\end{align*}
where $c$ and $c_b$ are constants defined in Lemma~\ref{lemma:expectationBoundOfX1} and Lemma~\ref{lemma:expectationBoundOfX3}.
\end{theorem}

\begin{proof} From Lemma~\ref{lemma:WassersteinFormula}, we have
\begin{align*}
\mathcal{W}_{q}^{q}(\mu_{1t}, \mu_{3t})=&\mathbb{E}\Big[\int_0^t q\,\Vert X_1(s)-X_3(s))\Vert^{q -2} \langle X_1(s)-X_3(s),b_1 (X_1(s-),\alpha)-b (X_3(s-),\alpha)\rangle\text{d}s\Big]\\
=&\int_{0}^{t}q\,\Vert X_1(s)-X_3(s))\Vert^{q -2} \langle X_1(s)-X_3(s),b_1 (X_1(s-),\alpha)-b (X_3(s-),\alpha)\rangle\text{d}s\\
\leq&\mathbb{E}\Big[\int_{0}^{t}q\,\Vert X_1(s)-X_3(s)\Vert^{q-1}\Vert c_\alpha\nabla f(X_1(s))+b (X_3(s),\alpha)\Vert\text{d}s\Big]\\
=&q\int_{0}^{t}\mathbb{E}\Big[\Vert X_1(s)-X_3(s)\Vert^{q-1}\Vert c_\alpha\nabla f(X_1(s))+b (X_3(s),\alpha)\Vert\Big]\text{d}s\\
\leq& q\int_{0}^{t}\Big[\mathbb{E}\Vert X_1(s)-X_3(s)\Vert^{(q-1)p_1}\Big]^{\frac{1}{p_1}}\Big[\mathbb{E}\Vert c_\alpha\nabla f(X_1(s))+b (X_3(s),\alpha)\Vert^{q_1}\Big]^{\frac{1}{q_1}}\text{d}s,
\end{align*}
where we have used Cauchy-Schwarz inequality in the third line and Holder's inequality in the last line.

Since $(q-1)p_1< 1$ by Assumption~\Cref{assump:conditionsOnParameters}, using Lemma \ref{lemma:anUsefulIneq} twice, we have:
\begin{align*}
\Big(\mathbb{E}\Vert X_1(s)-X_3(s)\Vert^{(q-1)p_1}\Big)^{\frac{1}{p_1}}\leq&\Big(\mathbb{E}\Vert X_1(s)\vert^{(q-1)p_1}+\mathbb{E}\Vert X_3(s)\Vert^{(q-1)p_1}\Big)^{\frac{1}{p_1}}\\
\leq&\Big[\mathbb{E}\Big(\Vert X_1(s)\Vert^{(q-1)p_1}\Big)\Big]^{\frac{1}{p_1}}+\Big[\mathbb{E}\Big(\Vert X_3(s)\Vert^{(q-1)p_1}\Big)\Big]^{\frac{1}{p_1}}
\end{align*}
Then, by applying Lemma~\ref{lemma:expectationBoundOfX1} and Lemma~\ref{lemma:expectationBoundOfX3} we obtain:
\begin{align*}
\Big(\mathbb{E}\Vert X_1(s)- & X_3(s)\Vert^{(q-1)p_1}\Big)^{\frac{1}{p_1}}\leq  \Big(c\Big(s\Big(\frac{d}{\beta^{1/\alpha}}+1\Big)+1\Big)\Big)^{q-1} + \Big(c_b\Big(s\Big(\frac{d}{\beta^{1/\alpha}}+1\Big)+1\Big)\Big)^{q-1}.
\end{align*}
Now, consider the following quantity
\begin{align*}
\Big[\mathbb{E}\Vert c_\alpha\nabla f(X_1(s))+b (X_3(s),\alpha)\Vert^{q_1}&\Big]^{\frac{1}{ q_1}}\\
\leq&\Big[\mathbb{E}\big(\Vert c_\alpha\nabla f(X_1(s))-c_\alpha\nabla f(X_3(s))\Vert+\Vert c_\alpha\nabla f(X_3(s))+b (X_3(s),\alpha)\Vert\big)^{q_1}\Big]^{\frac{1}{ q_1}}\\
\leq&\Big[\mathbb{E}\big(M\Vert X_1(s)-X_3(s)\Vert^{\gamma}+L\big)^{q_1}\Big]^{\frac{1}{ q_1}}\\
\leq&\Big[\mathbb{E}\Big(M\Vert X_1(s)\Vert^{\gamma}+ M\Vert X_3(s)\Vert^{\gamma} + L\Big)^{q_1}\Big]^{\frac{1}{q_1}}\\
\leq&\Big[\mathbb{E}\Big(M^{q_1}\Vert X_1(s)\Vert^{\gamma q_1}\Big)\Big]^{\frac{1}{q_1}}+\Big[\mathbb{E}\Big(M^{q_1}\Vert X_3(s)\Vert^{\gamma q_1}\Big)\Big]^{\frac{1}{q_1}} + L,
\end{align*}
where we have used Assumption~\Cref{assump:HolderContinuity}, Assumption~\Cref{assump:uniformlyBounded}, Lemma~\ref{lemma:anUsefulIneq} and Minkowski's inequality. By Lemma~\ref{lemma:expectationBoundOfX1} and Lemma~\ref{lemma:expectationBoundOfX3}, we have
\begin{align*}
\Big[\mathbb{E}\Vert c_\alpha\nabla f(X_1(s))+b (X_3(s),\alpha)\Vert^{q_1}\Big]^{\frac{1}{ q_1}}\leq& M\Big(c\Big(s\Big(\frac{d}{\beta^{1/\alpha}}+1\Big)+1\Big)\Big)^{\gamma} + M\Big(c_b\Big(s\Big(\frac{d}{\beta^{1/\alpha}}+1\Big)+1\Big)\Big)^{\gamma} + L.
\end{align*}
By combining the above inequalities, we get
\begin{align*}
\mathbb{E}\Big[\int_0^t q&\,\Vert X_1(s)-X_3(s))\Vert^{q -2} \langle X_1(s)-X_3(s),b_1 (X_1(s-),\alpha)-b (X_3(s-),\alpha)\rangle\text{d}s\Big]\\
\leq&q\int_{0}^{t}\Big(\Big(c\Big(s\Big(\frac{d}{\beta^{1/\alpha}}+1\Big)+1\Big)\Big)^{q-1} + \Big(c_b\Big(s\Big(\frac{d}{\beta^{1/\alpha}}+1\Big)+1\Big)\Big)^{q-1}\Big)\Big(M\Big(c\Big(s\Big(\frac{d}{\beta^{1/\alpha}}+1\Big)+1\Big)\Big)^{\gamma}\\
& + M\Big(c_b\Big(s\Big(\frac{d}{\beta^{1/\alpha}}+1\Big)+1\Big)\Big)^{\gamma} + L\Big)\text{d}s\\
=& q\int_{0}^{t}\Big(M(c^{q-1}+c_b^{q-1})(c^{\gamma}+c_b^{\gamma})\Big(s\Big(\frac{d}{\beta^{1/\alpha}}+1\Big)+1\Big)^{q-1+\gamma} + L(c^{q-1}+c_b^{q-1})\Big(s\Big(\frac{d}{\beta^{1/\alpha}}+1\Big)+1\Big)^{q-1}\Big)\text{d}s\\
\leq& qt\Big(M(c^{q-1}+c_b^{q-1})(c^{\gamma}+c_b^{\gamma})\Big(t\Big(\frac{d}{\beta^{1/\alpha}}+1\Big)+1\Big)^{q-1+\gamma} + L(c^{q-1}+c_b^{q-1})\Big(t\Big(\frac{d}{\beta^{1/\alpha}}+1\Big)+1\Big)^{q-1}\Big).
\end{align*}
The final conclusion follows from this inequality.
\end{proof}

\subsubsection{Proof of \Cref{cor:WassersteinOfX1X3}}
\begin{proof}
First, we replace $t$ by $k\eta$. Then, by following the same lines of the proof of Corollary~\ref{cor:BoundWasserstein}, we get
\begin{align*}
\mathcal{W}_{q}^{q}(\mu_{1t}, \mu_{3t})\leq& C(k^{q+\gamma}\eta + k^{q+\gamma}\eta^{q}\beta^{-\frac{q-1}{\alpha}}d^{q-1+\gamma}).
\end{align*}
By assumption \Cref{assump:conditionsOnParameters}, $q-1<1/p_1$ and $\gamma<1/q_1$. It implies that $d^{q-1+\gamma}<d^{1/p_1+1/q_1}=d$. Hence, we have
\begin{align*}
\mathcal{W}_{q}^{q}(\mu_{1t}, \mu_{3t})\leq& C(k^{q+\gamma}\eta + k^{q+\gamma}\eta^{q}\beta^{-\frac{q-1}{\alpha}}d).
\end{align*}
\end{proof}

\subsubsection{Proof of Corollary~\ref{cor:distanceFromX1toX3}}


\begin{proof}
By Lemma~\ref{lemma:DifferenceExpectation}, Lemma~\ref{lemma:expectationBoundOfX1} and Lemma~\ref{lemma:expectationBoundOfX3}, we have
\begin{align*}
c_\alpha\vert\mathbb{E}[f(X_1(t))] - \mathbb{E}[f(X_3(t))]\vert\leq& \Big(M\big(\mathbb{E}\Vert X_1(t)\Vert^{\gamma p}\big)^{\frac{1}{p}} + M\big(\mathbb{E}\Vert X_3(t)\Vert^{\gamma p}\big)^{\frac{1}{p}} + B\Big)\mathcal{W}_q(\mu_{1t},\mu_{3t})\\
\leq& \Big(M\Big(c\Big(t\Big(\frac{d}{\beta^{1/\alpha}}+1\Big)+1\Big)\Big)^{\gamma} + M\Big(c_b\Big(t\Big(\frac{d}{\beta^{1/\alpha}}+1\Big)+1\Big)\Big)^{\gamma} + B\Big)\mathcal{W}_q(\mu_{1t},\mu_{3t}).
\end{align*}
Then by Theorem~\ref{thm:WassersteinOfX1X3}, we have
\begin{align*}
c_\alpha\vert\mathbb{E}[f(X_1(t))] &- \mathbb{E}[f(X_3(t))]\vert\\
\leq& \Big(M\Big(c\Big(t\Big(\frac{d}{\beta^{1/\alpha}}+1\Big)+1\Big)\Big)^{\gamma} + M\Big(c_b\Big(t\Big(\frac{d}{\beta^{1/\alpha}}+1\Big)+1\Big)\Big)^{\gamma} + B\Big)\Bigg(qt\Big(M(c^{q-1}+c_b^{q-1})(c^{\gamma}+c_b^{\gamma})\\
&\Big(t\Big(\frac{d}{\beta^{1/\alpha}}+1\Big)+1\Big)^{q-1+\gamma} + L(c^{q-1}+c_b^{q-1})\Big(t\Big(\frac{d}{\beta^{1/\alpha}}+1\Big)+1\Big)^{q-1}\Big)\Bigg)^{\frac{1}{q}}.
\end{align*}
Applying Lemma~\ref{lemma:anUsefulIneq} twice, we get
\begin{align*}
c_\alpha\vert\mathbb{E}[f(X_1(t))]  - &\mathbb{E}[f(X_3(t))]\vert\\
\leq& \Big(M(c^{\gamma}+c_b^{\gamma})\Big(\frac{t^\gamma d^\gamma}{\beta^{\gamma/\alpha}}+t^\gamma+1\Big) + B\Big)\Bigg((qt)^{1/q}\Big(M^{1/q}(c^{q-1}+c_b^{q-1})^{1/q}(c^{\gamma}+c_b^{\gamma})^{1/q}\\
&\Big(\frac{td}{\beta^{1/\alpha}}+t+1\Big)^{(q-1+\gamma)/q} + L^{1/q}(c^{q-1}+c_b^{q-1})^{1/q}\Big(\frac{td}{\beta^{1/\alpha}}+t+1\Big)^{(q-1)/q}\Big)\Bigg)\\
\leq& \Big(M(c^{\gamma}+c_b^{\gamma})\Big(\frac{t^\gamma d^\gamma}{\beta^{\gamma/\alpha}}+t^\gamma +1\Big) + B\Big)\Bigg((qt)^{1/q}\Big(M^{1/q}(c^{q-1}+c_b^{q-1})^{1/q}(c^{\gamma}+c_b^{\gamma})^{1/q}\Big(\frac{(td)^{(q-1+\gamma)/q}}{\beta^{(q-1+\gamma)/(q\alpha)}}\\
&+t^{{(q-1+\gamma)/q}} +1\Big) + L^{1/q}(c^{q-1}+c_b^{q-1})^{1/q}\Big(\frac{(td)^{(q-1)/q}}{\beta^{(q-1)/(q\alpha)}}+t^{{(q-1)/q}} +1\Big)\Big)\Bigg).
\end{align*}
Now, by replacing $t=k\eta$ we find that, among the terms containing $\beta$, the largest power of $d$, the largest power of $k$ and the smallest power of $\eta$ are $\gamma + \frac{q-1+\gamma}{q}$, $\gamma+\frac{\gamma+q}{q}$ and $\gamma+\frac{1}{q}$, respectively. For the smallest power of $\beta^{-1}$, we need to compare the following quantities: $\gamma/\alpha$, $(q-1+\gamma)/(q\alpha)$ and $(q-1)/(q\alpha)$.

It is obvious that $(q-1+\gamma)/(q\alpha)>(q-1)/(q\alpha)$. Next, from the relation $\gamma<1/p=(q-1)/q$, we have $\gamma/\alpha<(q-1)/(q\alpha)$. Thus, the smallest power of $\beta^{-1}$ is $\gamma/\alpha$. Hence, we have the following bound:
\begin{align*}
c_\alpha\vert\mathbb{E}[f(X_1(t))]  - \mathbb{E}[f(X_3(t))]\vert\leq& C\Big(k^{\gamma+\frac{\gamma+q}{q}}\eta^{\gamma+\frac{1}{q}}\beta^{-\frac{\gamma}{\alpha}}d^{\gamma + \frac{q-1+\gamma}{q}} + k^{\gamma+\frac{\gamma + q}{q}}\eta^{\frac{1}{q}}\Big),
\end{align*}
for some constant $C>0$. For the power of $d$, using that $\gamma <1/p$, $q-1<1/p_1$ and $\gamma<1/q_1$ we have
\begin{align*}
\gamma + \frac{q-1+\gamma}{q}\leq& 1/p + \frac{1/p_1+1/q_1}{q}\\
=& 1/p+1/q\\
=&1.
\end{align*}
Finally, we have
\begin{align*}
c_\alpha\vert\mathbb{E}[f(X_1(t))]  - \mathbb{E}[f(X_3(t))]\vert\leq& C\Big(k^{\gamma+\frac{\gamma+q}{q}}\eta^{\gamma+\frac{1}{q}}\beta^{-\frac{\gamma}{\alpha}}d + k^{\gamma+\frac{\gamma + q}{q}}\eta^{\frac{1}{q}}\Big).
\end{align*}
\end{proof}

\subsection{Proof of Lemma~\ref{lemma:distanceFromX3ToInvariant2}}


\begin{proof}
By Lemma~\ref{lemma:DifferenceExpectation}, we have
\begin{align*}
c_\alpha\vert\mathbb{E}[f(X_3(t))] - \mathbb{E}[f(\hat{W})]\vert\leq& \Big(M\big(\mathbb{E}\Vert X_3(t)\Vert^{\gamma p}\big)^{\frac{1}{p}} + M\big(\mathbb{E}\Vert \hat{W}\Vert^{\gamma p}\big)^{\frac{1}{p}} + B\Big)\mathcal{W}_q(\mu_{3t},\pi).
\end{align*}
Assumption~\Cref{assump:momentsOfInvariant} says that $\mathbb{E}\Vert \hat{W}\Vert^{\gamma p}$ is bounded by a constant depending on $b,m$ and $\beta$. In addition, by Assumption~\Cref{assump:ergodicity}, $\lim_{t\rightarrow\infty}\mathcal{W}_{\gamma p}(\mu_{3t},\pi)=0$, and by Theorem 7.12 in \cite{villani2003topics}, it follows that
\begin{align*}
\lim_{t\rightarrow\infty}\mathbb{E}\Vert X_3(t)\Vert^{\gamma p}=\mathbb{E}\Vert \hat{W}\Vert^{\gamma p}.
\end{align*}
Thus, $\mathbb{E}\Vert X_3(t)\Vert^{\gamma p}$ is bounded by a constant independent of $t$. Finally, since $q<\alpha$, by Assumption~\Cref{assump:ergodicity} again, $\mathcal{W}_{q}(\mu_{3t},\pi)\leq C\beta e^{-\lambda_*t/\beta}$. Hence, using the bound in Assumption~\Cref{assump:momentsOfInvariant}, there exists constant $C$ such that
\begin{align*}
\vert\mathbb{E}[f(X_3(t))] - \mathbb{E}[f(\hat{W})]\vert\leq C\beta\frac{b+d/\beta}{m}\exp(- \lambda_*\beta^{-1}t).
\end{align*}
\end{proof}

\subsection{Proof of Lemma~\ref{lem:distanceInvariantToMinimumInExpectation}}


\begin{proof}
The proof is adapted from \cite{raginsky17a}, Section 3.5. First, we have the decomposition:
\begin{align*}
\mathbb{E}[f(\hat{W})]=&\int_{\mathbb{R}^d}f(w)\frac{\exp(-\beta f(w))}{\int_{\mathbb{R}^d}\exp(-\beta f(v))\text{d}v}\text{d}w\\
=&\frac{1}{\beta}\Big(-\int_{\mathbb{R}^d}\frac{\exp(-\beta f(w))}{\int_{\mathbb{R}^d}\exp(-\beta f(v))\text{d}v}\log\frac{\exp(-\beta f(w))}{\int_{\mathbb{R}^d}\exp(-\beta f(v))\text{d}v}\text{d}w - \log\int_{\mathbb{R}^d}\exp(-\beta f(v))\text{d}v\Big).
\end{align*}
The first term in the parentheses is the differential entropy of the probability density of $\hat{W}$, which has a finite second moment (due to Assumption~\Cref{assump:momentsOfInvariant}). Hence, it is upper-bounded by the differential entropy of a Gaussian density with the same second moment:
\begin{align*}
-\int_{\mathbb{R}^d}\frac{\exp(-\beta f(w))}{\int_{\mathbb{R}^d}\exp(-\beta f(v))\text{d}v}\log\frac{\exp(-\beta f(w))}{\int_{\mathbb{R}^d}\exp(-\beta f(v))\text{d}v}\text{d}w\leq\frac{d}{2}\log\Big(\frac{2\boldsymbol{\pi} e (b+d/\beta)}{dm}\Big).
\end{align*}
By Lemma~\ref{lemma:lowerBoundForNormalizedInvariant}, we have
\begin{align*}
-\log\int_{\mathbb{R}^d}\exp(-\beta f(w))\text{d}w\leq\beta f(w^*)+\frac{\beta^{-\gamma} Mc_\alpha^{-1}}{1+\gamma} - \log\Big(\frac{\boldsymbol{\pi}^{d/2}\beta^{-d}}{\Gamma(d/2+1)}\Big).
\end{align*}
Then, it implies that
\begin{align*}
\mathbb{E}[f(\hat{W})]\leq& \frac{d\beta^{-1}}{2}\log\Big(\frac{2\boldsymbol{\pi} e (b+d/\beta)}{dm}\Big) + f(w^*) + \frac{\beta^{-\gamma-1} Mc_\alpha^{-1}}{1+\gamma} - \beta^{-1}\log\Big(\frac{\boldsymbol{\pi}^{d/2}\beta^{-d}}{\Gamma(d/2+1)}\Big)\\
=& f(w^*) + \frac{\beta^{-\gamma-1} Mc_\alpha^{-1}}{1+\gamma} + \beta^{-1}\log\Big(\frac{(2e(b+d/\beta))^{d/2}\Gamma(d/2+1)\beta^d}{(dm)^{d/2}}\Big),
\end{align*}
which leads to desired result.
\end{proof}

\subsection{Proof of \Cref{thm:WassersteinForSampling}}
\begin{proof}
By triangular inequality, we have
\begin{align*}
\mathcal{W}_{q}(\mu_{2t}, \pi)\leq\mathcal{W}_{q}(\mu_{2t}, \mu_{1t})+\mathcal{W}_{q}(\mu_{1t}, \mu_{3t})+\mathcal{W}_{q}(\mu_{3t}, \pi).
\end{align*}
Then, using \Cref{cor:BoundWasserstein}, \Cref{cor:WassersteinOfX1X3} and assumption \Cref{assump:ergodicity}, we get
\begin{align*}
\mathcal{W}_{q}(\mu_{2t}, \pi)\leq& C\Big((k^2\eta + k^2\eta^{1+\gamma/\alpha}\beta^{-\gamma(q-1)/\alpha}d)^{1/q} + (k^{q+\gamma}\eta + k^{q+\gamma}\eta^{q}\beta^{-(q-1)/\alpha}d)^{1/q} + \beta e^{- \lambda_* k\eta /\beta}\Big)\\
\leq& C\Big(k^{2/q}\eta^{1/q} + k^{2/q}\eta^{1/q+\gamma/(q\alpha)}\beta^{-\gamma(q-1)/(q\alpha)}d^{1/q} + k^{1+\gamma/q}\eta^{1/q} + k^{1+\gamma/q}\eta\beta^{-(q-1)/(q\alpha)}d^{1/q}\\
& + \beta e^{- \lambda_* k\eta /\beta}\Big),
\end{align*}
where, we have used \Cref{lemma:anUsefulIneq} for the second inequality. Then, similar to the proof of \Cref{cor:BoundWasserstein}, we obtain
\begin{align*}
\mathcal{W}_{q}(\mu_{2t}, \pi)\leq& C\Big(k^{\max\{2,q+\gamma\}/q}\eta^{1/q} + k^{\max\{2,q+\gamma\}/q}\eta^{1/q+\gamma/(q\alpha)}\beta^{-\gamma(q-1)/(q\alpha)}d^{1/q} + \beta e^{- \lambda_* k\eta /\beta}\Big).
\end{align*}
\end{proof}

\subsection{Proof of \Cref{thm:StochGradForX1X2}}
\begin{proof}
Since each function $x\mapsto f^{(i)}(x)$ satisfies assumptions \Cref{assump:boundedGradAtZero}-\Cref{assump:momentsOfInvariant}, it is easy to check that $f_k$ also satisfies these assumptions (with the same constants and the same parameters) for all $k$. Then by repeating exactly the same lines as in the proof of \Cref{lemma:expecataionBound}, we obtain the same estimates for the moments of $X_2$. Now by following the same steps as in the proof of \Cref{thm:boundOfWasserstein}, we first have
\begin{align*}
\mathcal{W}_{q}^{q}(\mu_{1t}, \mu_{2t})\leq& q\,\sum_{j=0}^{k-1}\int_{j\eta}^{(j+1)\eta}\Big[\mathbb{E}\Vert X_1(s)-X_2(s)\Vert^{(q-1)p_1}\Big]^{\frac{1}{p_1}}\Big[\mathbb{E}\Vert c_\alpha(\nabla f(X_1(s))-\nabla f_k(X_2(j\eta)))\Vert^{q_1}\Big]^{\frac{1}{q_1}}\text{d}s,
\end{align*}
then
\begin{align*}
\Big(\mathbb{E}\Vert X_1(s)&-X_2(s)\Vert^{(q-1)p_1} \Big)^{\frac{1}{p_1}}\\
\leq& c^{\frac{1}{p_1}} + (\mathbb{E}\Vert X_2(0)\Vert^{(q-1)p_1})^{\frac{1}{p_1}} + \eta^{q-1}\Big(M^{q-1}(\mathbb{E}\Vert X_2(0)\Vert^{(q-1)p_1\gamma})^{\frac{1}{p_1}} + B^{(q-1)} \Big) + \Big(\frac{\eta}{\beta}\Big)^{\frac{q-1}{\alpha}}l_{\alpha,(q-1)p_1,d}^{\frac{1}{p_1}}\\
& + (j+1)^{\frac{1}{p_1}}\Big(\Big(c\eta\Big(\frac{d}{\beta^{1/\alpha}}\Big)\Big)^{\frac{1}{p_1}} + (c\eta)^{\frac{1}{p_1}} + (2\eta (b+m))^{\frac{(q-1)}{2}} + 2^{\frac{(q-1)}{2}}(\eta B)^{(q-1)} +\Big(\frac{\eta}{\beta}\Big)^{\frac{(q-1)}{\alpha}}l_{\alpha,(q-1)p_1,d}^{\frac{1}{p_1}}\\
&+ \eta^{q-1} M^{q-1} \Big((2\eta (b+m))^{\frac{(q-1)\gamma}{2}} + 2^{\frac{(q-1)\gamma}{2}}(\eta B)^{(q-1)\gamma} +\Big(\frac{\eta}{\beta}\Big)^{\frac{(q-1)\gamma}{\alpha}}l_{\alpha,(q-1)p_1\gamma,d}^{\frac{1}{p_1}}\Big)\Big)\\
=&Q_1(\eta) + (j+1)^{\frac{1}{p_1}}P_1(\eta),
\end{align*}
where $P_1(\eta)$ and $Q_1(\eta)$ are defined in \Cref{thm:boundOfWasserstein}. Now, by Minkowski's inequality, we have
\begin{align*}
\Big[\mathbb{E}\Vert c_\alpha(\nabla f(X_1(s))-\nabla f_k(X_2(j\eta)))\Vert^{q_1}\Big]^{\frac{1}{ q_1}}=&\Big[\mathbb{E}\Vert c_\alpha(\nabla f(X_1(s))-\nabla f(X_2(j\eta)) + \nabla f(X_2(j\eta))\\
&- \nabla f_k(X_2(j\eta)))\Vert^{q_1}\Big]^{\frac{1}{ q_1}}\\
\leq&\Big[\mathbb{E}\Vert c_\alpha(\nabla f(X_1(s))-\nabla f(X_2(j\eta)))\Vert^{q_1}\Big]^{\frac{1}{q_1}} + \Big[\mathbb{E}\Vert c_\alpha(\nabla f(X_2(j\eta))\\
&- \nabla f_k(X_2(j\eta)))\Vert^{q_1}\Big]^{\frac{1}{ q_1}}.
\end{align*}
As in the proof of \Cref{thm:boundOfWasserstein}, the following inequality holds:
\begin{align*}
\Big[\mathbb{E}\Vert c_\alpha\nabla f(X_1(s))-c_\alpha\nabla f(X_2(j\eta))\Vert^{q_1}\Big]^{\frac{1}{ q_1}}\leq& M(\mathbb{E}\Vert X_2(0)\Vert^{\gamma q_1})^{\frac{1}{q_1}} + Mc^{\frac{1}{q_1}} + M(j+1)^{\frac{1}{q_1}}\Big(\Big(c\eta\Big(\frac{d}{\beta^{1/\alpha}}\Big)\Big)^{\frac{1}{q_1}}\\
& + (c\eta)^{\frac{1}{q_1}} + (2\eta (b+m))^{\frac{\gamma}{2}} +2^{\frac{\gamma}{2}}(\eta B)^{\gamma} +\Big(\frac{\eta}{\beta}\Big)^{\frac{\gamma}{\alpha}}l_{\alpha,\gamma q_1,d}^{\frac{1}{q_1}}\Big)\\
=&Q_2 + (j+1)^{\frac{1}{q_1}}P_2(\eta),
\end{align*}
where $P_2(\eta)$ and $Q_2$ are defined in \Cref{thm:boundOfWasserstein}. Using the additional assumption, \Cref{lemma:expecataionBound}, and \Cref{lemma:anUsefulIneq}, we get
\begin{align*}
\Big[\mathbb{E}\Vert c_\alpha(\nabla f(X_2(j\eta)) - \nabla f_k(X_2(j\eta)))\Vert^{q_1}\Big]^{\frac{1}{ q_1}}\leq&\delta\Big[\mathbb{E}\Big(M^{q_1}\Vert X_2(j\eta)\Vert^{\gamma q_1}\Big)\Big]^{\frac{1}{q_1}}\\
\leq& \delta\Big[M^{q_1}(\mathbb{E}\Vert X_2(0)\Vert^{\gamma q_1}) + M^{q_1}j\Big((2\eta (b+m))^{\frac{\gamma q_1}{2}} +2^{\frac{\gamma q_1}{2}}(\eta B)^{\gamma q_1}\\
& +\Big(\frac{\eta}{\beta}\Big)^{\frac{\gamma q_1}{\alpha}}l_{\alpha,\gamma q_1,d}\Big)\Big]^{\frac{1}{q_1}}\\
\leq& \delta M(\mathbb{E}\Vert X_2(0)\Vert^{\gamma q_1})^{\frac{1}{q_1}} + \delta M(j+1)^{\frac{1}{q_1}}\Big((2\eta (b+m))^{\frac{\gamma }{2}} +2^{\frac{\gamma }{2}}(\eta B)^{\gamma }\\
& +\Big(\frac{\eta}{\beta}\Big)^{\frac{\gamma }{\alpha}}l_{\alpha,\gamma q_1,d}^{\frac{1}{q_1}}\Big).
\end{align*}
By combining the two above inequalities, we obtain
\begin{align*}
\Big[\mathbb{E}\Vert c_\alpha\nabla f(X_1(s))-c_\alpha\nabla f(X_2(j\eta))\Vert^{q_1}\Big]^{\frac{1}{ q_1}}\leq& (1+\delta)M(\mathbb{E}\Vert X_2(0)\Vert^{\gamma q_1})^{\frac{1}{q_1}} + Mc^{\frac{1}{q_1}} + M(j+1)^{\frac{1}{q_1}}\Big(\Big(c\eta\Big(\frac{d}{\beta^{1/\alpha}}\Big)\Big)^{\frac{1}{q_1}}\\
& + (c\eta)^{\frac{1}{q_1}} + (1+\delta)(2\eta (b+m))^{\frac{\gamma}{2}} +(1+\delta)2^{\frac{\gamma}{2}}(\eta B)^{\gamma}\\
& +(1+\delta)\Big(\frac{\eta}{\beta}\Big)^{\frac{\gamma}{\alpha}}l_{\alpha,\gamma q_1,d}^{\frac{1}{q_1}}\Big)\\
=&Q'_2 + (j+1)^{\frac{1}{q_1}}P'_2(\eta).
\end{align*}
Finally, we have
\begin{align*}
\mathcal{W}_{q}^{q}(\mu_{1t}, \mu_{2t})\leq q\eta\Big(k^2P_1(\eta)P'_2(\eta) + k^{1+1/p_1}P_1(\eta)Q'_2 + k^{1+1/q_1}P'_2(\eta)Q_1(\eta) + kQ_1(\eta)Q'_2\Big).
\end{align*}
By considering the additional term $\delta$, we arrive at the following bound:
\begin{align*}
\mathcal{W}_{q}^{q}(\mu_{1t}, \mu_{2t})\leq C(1+\delta)(k^2\eta + k^2\eta^{1+\gamma/\alpha}\beta^{-\gamma(q-1)/\alpha}d).
\end{align*}
\end{proof}

\subsection{Proof of \Cref{cor:StochGradForX1X2}}
\begin{proof}
By \Cref{lemma:DifferenceExpectation},
\begin{align*}
c_\alpha\big\vert\mathbb{E}[f(X_1(k\eta))] - \mathbb{E}[f(X_2(k\eta))]\big\vert\leq \Big(M\Big(\mathbb{E}_{\mathbf{P}}\Vert X_1(k\eta)\Vert^{\gamma p}\Big)^{\frac{1}{p}}+M\Big(\mathbb{E}_{\mathbf{P}}\Vert X_2(k\eta)\Vert^{\gamma p}\Big)^{\frac{1}{p}}+B\Big)\mathcal{W}_{q}(\mu_{1t}, \mu_{2t}).
\end{align*}
Then, by following the same proof as in \Cref{cor:distanceFromX1toX2}, \Cref{cor:BoundWasserstein} and using \Cref{thm:StochGradForX1X2}, we get
\begin{align*}
c_\alpha\big\vert\mathbb{E}[f(X_1(k\eta))] - \mathbb{E}[f(X_2(k\eta))]\big\vert\leq& C(1+\delta)\Big(k^{1+\frac{1}{q}}\eta^{\frac{1}{q}}  +  k^{1+\frac{1}{q}}\eta^{\frac{1}{q}+\frac{\gamma}{\alpha q}}\beta^{-\frac{(q-1)\gamma}{\alpha q}}d\Big).
\end{align*}
\end{proof}

\subsection{Technical Results}

\begin{corollary}\label{cor:distanceFromX1toX2} Along with $P_1(\eta),P_2(\eta),Q_1(\eta),Q_2$ in Lemma~\ref{thm:boundOfWasserstein}, we define, in addition, the following quantities:
\begin{align*}
&P_3(\eta)\triangleq M\Big(\Big(c\eta\Big(\frac{d}{\beta^{1/\alpha}}\Big)\Big)^{\frac{1}{p}} + (c\eta)^{\frac{1}{p}} + (2\eta (b+m))^{\frac{\gamma}{2}}+2^{\frac{\gamma}{2}}(\eta B)^{\gamma} +\Big(\frac{\eta}{\beta}\Big)^{\frac{\gamma}{\alpha}}l_{\alpha,\gamma p,d}^{\frac{1}{p}}\Big)\\
&Q_3\triangleq M(\mathbb{E}\Vert X_2(0)\Vert^{\gamma p})^{\frac{1}{p}} + Mc^{\frac{1}{p}} + B.
\end{align*}
For $0<\eta<\frac{m}{M^2}$, we have the following bound:
\begin{align*}
c_\alpha\big\vert\mathbb{E}[f(X_1(k\eta))] - &\mathbb{E}[f(X_2(k\eta))]\big\vert\\
\leq&(q\eta)^{\frac{1}{q}}\Big(k^{1+\frac{1}{q}}(P_1(\eta)P_2(\eta))^{\frac{1}{q}}P_3(\eta) + k^{1+\frac{1}{qp_1}}(P_1(\eta)Q_2)^{\frac{1}{q}}P_3(\eta) + k^{1+\frac{1}{qq_1}}(P_2(\eta)Q_1(\eta))^{\frac{1}{q}}P_3(\eta)\\
& + k(Q_1(\eta)Q_2)^{\frac{1}{q}}P_3(\eta) +k^{\frac{2}{q}}(P_1(\eta)P_2(\eta))^{\frac{1}{q}}Q_3 + k^{\frac{1}{q}+\frac{1}{qp_1}}(P_1(\eta)Q_2)^{\frac{1}{q}}Q_3\\
& + k^{\frac{1}{q}+\frac{1}{qq_1}}(P_2(\eta)Q_1(\eta))^{\frac{1}{q}}Q_3 + k^{\frac{1}{q}}(Q_1(\eta)Q_2)^{\frac{1}{q}}Q_3\Big).
\end{align*}
\end{corollary}

\begin{proof} By Lemma~\ref{lemma:DifferenceExpectation},
\begin{align*}
c_\alpha\big\vert\mathbb{E}[f(X_1(k\eta))] - \mathbb{E}[f(X_2(k\eta))]\big\vert\leq \Big(M\Big(\mathbb{E}_{\mathbf{P}}\Vert X_1(k\eta)\Vert^{\gamma p}\Big)^{\frac{1}{p}}+M\Big(\mathbb{E}_{\mathbf{P}}\Vert X_2(k\eta)\Vert^{\gamma p}\Big)^{\frac{1}{p}}+B\Big)\mathcal{W}_{q}(\mu_{1t}, \mu_{2t}).
\end{align*}
Using Lemma~\ref{lemma:expectationBoundOfX1} and Lemma~\ref{lemma:gradientBound}, we have
\begin{align*}
\Big(M\Big(\mathbb{E}_{\mathbf{P}}\Vert X_1(k\eta)\Vert^{\gamma p}\Big)^{\frac{1}{p}}+M\Big(\mathbb{E}_{\mathbf{P}}\Vert X_2(k\eta)\Vert^{\gamma p}\Big)^{\frac{1}{p}}+B\Big)\leq& M\Big(c\Big(k\eta\Big(\frac{d}{\beta^{1/\alpha}}+1\Big)+1\Big)\Big)^{\gamma} + M\Big[(\mathbb{E}\Vert X_2(0)\Vert^{\gamma p})\\
& + k\Big((2\eta (b+m))^{\frac{\gamma p}{2}} +2^{\frac{\gamma p}{2}}(\eta B)^{\gamma p} +\Big(\frac{\eta}{\beta}\Big)^{\frac{\gamma p}{\alpha}}l_{\alpha,\gamma p,d}\Big)\Big]^{\frac{1}{p}}\\
& + B.
\end{align*}
By using Lemma~\ref{lemma:anUsefulIneq}, we obtain: $\Big(M\Big(\mathbb{E}_{\mathbf{P}}\Vert X_1(k\eta)\Vert^{\gamma p}\Big)^{\frac{1}{p}}+M\Big(\mathbb{E}_{\mathbf{P}}\Vert X_2(k\eta)\Vert^{\gamma p}\Big)^{\frac{1}{p}}+B\Big)\leq$
\begin{align*}
& M\Big(c\Big(k\eta\Big(\frac{d}{\beta^{1/\alpha}} + 1\Big)+1\Big)\Big)^{\gamma} + M(\mathbb{E}\Vert X_2(0)\Vert^{\gamma p})^{\frac{1}{p}} + Mk^{\frac{1}{p}}\Big((2\eta (b+m))^{\frac{\gamma }{2}} +2^{\frac{\gamma }{2}}(\eta B)^{\gamma } +\Big(\frac{\eta}{\beta}\Big)^{\frac{\gamma }{\alpha}}l_{\alpha,\gamma p,d}^{\frac{1}{p}}\Big) + B.
\end{align*}
We note that $\gamma<\frac{1}{p}$. Hence,
\begin{align*}
\Big(c\Big(k\eta\Big(\frac{d}{\beta^{1/\alpha}} + 1\Big)+1\Big)\Big)^{\gamma}\leq& \Big(c\Big(k\eta\Big(\frac{d}{\beta^{1/\alpha}} + 1\Big)+1\Big)\Big)^{\frac{1}{p}}\\
\leq& k^{\frac{1}{p}}\Big(c\eta\Big(\frac{d}{\beta^{1/\alpha}} + 1\Big)\Big)^{\frac{1}{p}} + c^{\frac{1}{p}},
\end{align*}
where the last inequality is an application of Lemma~\ref{lemma:anUsefulIneq}. By replacing this inequality into the previous one and rearranging the terms, we have

\begin{align*}
\Big(M\Big(\mathbb{E}_{\mathbf{P}}\Vert X_1(k\eta)\Vert^{\gamma p}\Big)^{\frac{1}{p}}+M\Big(\mathbb{E}_{\mathbf{P}}\Vert X_2(k\eta)\Vert^{\gamma p}\Big)^{\frac{1}{p}}&+B\Big)\\
\leq& M(\mathbb{E}\Vert X_2(0)\Vert^{\gamma p})^{\frac{1}{p}} + Mc^{\frac{1}{p}} + B + Mk^{\frac{1}{p}}\Big(\Big(c\eta\Big(\frac{d}{\beta^{1/\alpha}} + 1\Big)\Big)^{\frac{1}{p}} \\
& + (2\eta (b+m))^{\frac{\gamma}{2}} +2^{\frac{\gamma}{2}}(\eta B)^{\gamma} +\Big(\frac{\eta}{\beta}\Big)^{\frac{\gamma}{\alpha}}l_{\alpha,\gamma p,d}^{\frac{1}{p}}\Big)\\
\leq& M(\mathbb{E}\Vert X_2(0)\Vert^{\gamma p})^{\frac{1}{p}} + Mc^{\frac{1}{p}} + B + Mk^{\frac{1}{p}}\Big(\Big(c\eta\Big(\frac{d}{\beta^{1/\alpha}}\Big)\Big)^{\frac{1}{p}} \\
& + (c\eta)^{\frac{1}{p}} + (2\eta (b+m))^{\frac{\gamma}{2}} +2^{\frac{\gamma}{2}}(\eta B)^{\gamma} +\Big(\frac{\eta}{\beta}\Big)^{\frac{\gamma}{\alpha}}l_{\alpha,\gamma p,d}^{\frac{1}{p}}\Big)\\
=&Q_3 + k^{\frac{1}{p}}P_3(\eta).
\end{align*}
Here, we have used \Cref{lemma:anUsefulIneq} in the last inequality. Next, by Lemma~\ref{thm:boundOfWasserstein} and Lemma~\ref{lemma:anUsefulIneq},
\begin{align*}
\mathcal{W}_{q}(\mu_{1t},\mu_{2t})\leq& (q\eta)^{\frac{1}{q}}\Big(k^2P_1(\eta)P_2(\eta) + k^{1+1/p_1}P_1(\eta)Q_2 + k^{1+1/q_1}P_2(\eta)Q_1(\eta) + kQ_1(\eta)Q_2\Big)^{\frac{1}{q}}\\
\leq& (q\eta)^{\frac{1}{q}}\Big(k^{\frac{2}{q}}(P_1(\eta)P_2(\eta))^{\frac{1}{q}} + k^{\frac{1}{q}+\frac{1}{qp_1}}(P_1(\eta)Q_2)^{\frac{1}{q}} + k^{\frac{1}{q}+\frac{1}{qq_1}}(P_2(\eta)Q_1(\eta))^{\frac{1}{q}} + k^{\frac{1}{q}}(Q_1(\eta)Q_2)^{\frac{1}{q}}\Big).
\end{align*}
By combining the above two inequalities, we get
\begin{align*}
c_\alpha \big\vert\mathbb{E}&[f(X_1(k\eta))] - \mathbb{E}[f(X_2(k\eta))]\big\vert\\
\leq& (q\eta)^{\frac{1}{q}}\Big(Q_3 + k^{\frac{1}{p}}P_3(\eta)\Big)\Big(k^{\frac{2}{q}}(P_1(\eta)P_2(\eta))^{\frac{1}{q}} + k^{\frac{1}{q}+\frac{1}{qp_1}}(P_1(\eta)Q_2)^{\frac{1}{q}} + k^{\frac{1}{q}+\frac{1}{qq_1}}(P_2(\eta)Q_1(\eta))^{\frac{1}{q}} + k^{\frac{1}{q}}(Q_1(\eta)Q_2)^{\frac{1}{q}}\Big)\\
=& (q\eta)^{\frac{1}{q}}\Big(k^{1+\frac{1}{q}}(P_1(\eta)P_2(\eta))^{\frac{1}{q}}P_3(\eta) + k^{1+\frac{1}{qp_1}}(P_1(\eta)Q_2)^{\frac{1}{q}}P_3(\eta) + k^{1+\frac{1}{qq_1}}(P_2(\eta)Q_1(\eta))^{\frac{1}{q}}P_3(\eta)\\
& + k(Q_1(\eta)Q_2)^{\frac{1}{q}}P_3(\eta) +k^{\frac{2}{q}}(P_1(\eta)P_2(\eta))^{\frac{1}{q}}Q_3 + k^{\frac{1}{q}+\frac{1}{qp_1}}(P_1(\eta)Q_2)^{\frac{1}{q}}Q_3 + k^{\frac{1}{q}+\frac{1}{qq_1}}(P_2(\eta)Q_1(\eta))^{\frac{1}{q}}Q_3\\
& + k^{\frac{1}{q}}(Q_1(\eta)Q_2)^{\frac{1}{q}}Q_3\Big).
\end{align*}

\end{proof}

The following lemma is an extension of Lemma 1.2.3 in \cite{nesterov2013introductory} to functions with H\"{o}lder continuous gradients.

\begin{lemma}\label{lemma:smoothnessInequality}
Under Assumption~\Cref{assump:HolderContinuity}, the following inequality holds for any $x,y\in\mathbb{R}^d$:
\begin{align*}
c_\alpha\vert f(x)-f(y)-\langle\nabla f(y),x-y\rangle\vert\leq\frac{M}{1+\gamma}\Vert x-y\Vert^{1+\gamma}.
\end{align*}
\end{lemma}

\begin{proof}
Let $g(t)\triangleq c_\alpha f(y+t(x-y))$. Then, $g'(t)=c_\alpha\langle \nabla f(y+t(x-y)),x-y\rangle$ and $\int_0^1g'(t)\text{d}t=g(1)-g(0)=c_\alpha(f(x)-f(y))$. We have
\begin{align*}
c_\alpha\vert f(x)-f(y)-\langle\nabla f(y),x-y\rangle\vert=&\Big\vert\int_0^1g'(t)\text{d}t - c_\alpha\langle\nabla f(y),x-y\rangle\Big\vert\\
=&\Big\vert\int_0^1 c_\alpha\langle \nabla f(y+t(x-y)),x-y\rangle\text{d}t - c_\alpha\langle\nabla f(y),x-y\rangle\Big\vert\\
=&\Big\vert\int_0^1 c_\alpha\langle \nabla f(y+t(x-y))- \nabla f(y),x-y\rangle\text{d}t\Big\vert.
\end{align*}
By Cauchy-Schwarz inequality and Assumption~\Cref{assump:HolderContinuity}, we have
\begin{align*}
c_\alpha\vert f(x)-f(y)-\langle\nabla f(y),x-y\rangle\vert\leq&\int_0^1 c_\alpha\Vert \nabla f(y+t(x-y))- \nabla f(y)\Vert\Vert x-y\Vert\text{d}t\\
\leq&\int_0^1 Mt^{\gamma}\Vert x-y\Vert^{\gamma}\Vert x-y\Vert\text{d}t\\
=&\frac{M}{1+\gamma}\Vert x-y\Vert^{1+\gamma}.
\end{align*}
\end{proof}

\begin{lemma}\label{lemma:lowerBoundForNormalizedInvariant} The normalized factor of $\pi$ is bounded below, i. e.,
\begin{align*}
\log\int_{\mathbb{R}^d}\exp(-\beta f(w))\text{d}w\geq-\beta f(w^*)-\frac{\beta^{-\gamma} Mc_\alpha^{-1}}{1+\gamma} + \log\Big(\frac{\boldsymbol{\pi}^{d/2}\beta^{-d}}{\Gamma(d/2+1)}\Big).
\end{align*}
\end{lemma}

\begin{proof} We start by writing:
\begin{align*}
\log\int_{\mathbb{R}^d}\exp(-\beta f(w))\text{d}w=&-\beta f(w^*)+\log\int_{\mathbb{R}^d}\exp\big(-\beta(f(w)-f(w^*))\big)\text{d}w\\
\geq& -\beta f(w^*)+\log\int_{\mathbb{R}^d}\exp\Big(-\frac{\beta Mc_\alpha^{-1}}{1+\gamma}\Vert w-w^*\Vert^{1+\gamma}\Big)\text{d}w.
\end{align*}
Here, we used Lemma~\ref{lemma:smoothnessInequality}, with $\nabla f(w^*)=0$. For the second term on the right hand side, we have
\begin{align*}
\int_{\mathbb{R}^d}\exp\Big(-\frac{\beta Mc_\alpha^{-1}}{1+\gamma}\Vert w-w^*\Vert^{1+\gamma}\Big)\text{d}w=&\int_{\Vert w\Vert\leq\beta^{-1}}\exp\Big(-\frac{\beta Mc_\alpha^{-1}}{1+\gamma}\Vert w\Vert^{1+\gamma}\Big)\text{d}w  \\
&+ \int_{\Vert w\Vert\geq\beta^{-1}}\exp\Big(-\frac{\beta Mc_\alpha^{-1}}{1+\gamma}\Vert w\Vert^{1+\gamma}\Big)\text{d}w\\
\geq& \int_{\Vert w\Vert\leq\beta^{-1}}\exp\Big(-\frac{\beta Mc_\alpha^{-1}}{1+\gamma}\beta^{-1-\gamma}\Big)\text{d}w + 0\\
=& \exp\Big(-\frac{\beta^{-\gamma} Mc_\alpha^{-1}}{1+\gamma}\Big)\int_{\Vert w\Vert\leq\beta^{-1}}1\text{d}w\\
=&\exp\Big(-\frac{\beta^{-\gamma} Mc_\alpha^{-1}}{1+\gamma}\Big)\frac{\boldsymbol{\pi}^{d/2}\beta^{-d}}{\Gamma(d/2+1)},
\end{align*}
where, $\Gamma$ denotes the Gamma function and $\boldsymbol{\pi}$ denotes  Archimedes' constant (here, it is not the invariant distribution). Hence,
\begin{align*}
\log\int_{\mathbb{R}^d}\exp\Big(-\frac{\beta Mc_\alpha^{-1}}{1+\gamma}\Vert w-w^*\Vert^{1+\gamma}\Big)\text{d}w\geq&-\frac{\beta^{-\gamma} Mc_\alpha^{-1}}{1+\gamma} + \log\Big(\frac{\boldsymbol{\pi}^{d/2}\beta^{-d}}{\Gamma(d/2+1)}\Big).
\end{align*}
By combining the above inequalities, we have the desired result.
\end{proof}

\begin{lemma}\label{lemma:expectationBoundOfX1}
For $\lambda\in(0,1)$, there exists a constant $c$ depending on $m,b,\alpha$, such that
\begin{align*}
\mathbb{E}\Big( \Vert X_1(t)\Vert^{\lambda}\Big)^{\frac{1}{\lambda}}\leq c\Big(t(d\beta^{-1/\alpha}+1)+1\Big),&& \forall t>0, \beta\geq1, 1<\alpha<2.
\end{align*}
\end{lemma}
\begin{proof}
We follow exactly the same proof as Lemma 7.1 in \cite{xie2017ergodicity}, with some modifications. Let $h(x)\triangleq(1+\Vert x\Vert^2)^{1/2}$. By It\^{o}'s formula, we have $\text{d}h(X_1(t))=\Bigg(\langle b_1(X_1(t)),\nabla h(X_1(t))\rangle$
\begin{align}\label{eqn:ItoFormulaForH}
 + \int_{\mathbb{R}^d}\Big(h(X_1(t)+\beta^{-1/\alpha} x) - h(X_1(t)) - \mathbb{I}_{\Vert x\Vert<1}\langle\beta^{-1/\alpha} x,\nabla h(X_1(t))\rangle\Big)\nu(\text{d}x)\Bigg)\text{d}t + \text{d}M(t),
\end{align}
where $M(t)$ is a local martingale. Noticing that $\partial_ih(x)=x_i(1+\Vert x\Vert^2)^{-1/2}/2$ and using Assumption~\Cref{assump:dissipative}, we have
\begin{align*}
\langle b_1(x),\nabla h(x)\rangle =&\langle b_1(x),x\rangle (1+\Vert x\Vert^2)^{-1/2}/2\\
\leq& (-m\Vert x\Vert^{1+\gamma}+b)(1+\Vert x\Vert^2)^{-1/2}/2\\
=&(-m(\Vert x\Vert^{1+\gamma}+1)+m+b)(1+\Vert x\Vert^2)^{-1/2}/2.
\end{align*}
Since $(\Vert x\Vert^{2}+1)^{(1+\gamma)/2}\leq(\Vert x\Vert^{1+\gamma}+1)$ by Lemma~\ref{lemma:anUsefulIneq}, it follows that
\begin{align*}
\langle b_1(x),\nabla h(x)\rangle\leq&(-m(\Vert x\Vert^{2}+1)^{(1+\gamma)/2}+m+b)(1+\Vert x\Vert^2)^{-1/2}/2\\
=&(-m(\Vert x\Vert^{2}+1)^{\gamma/2}+(m+b)(1+\Vert x\Vert^2)^{-1/2})/2\\
\leq&(-m(\Vert x\Vert^{2}+1)^{\gamma/2}+m+b)/2\\
=&(-mh(x)^{\gamma}+m+b)/2.
\end{align*}
On the other hand, observing that
\begin{align*}
\vert h(x+y)-h(x)\vert\leq\Vert y\Vert\int_0^1\Vert\nabla h(x+sy)\Vert\text{d}s\leq\Vert y\Vert/2,
\end{align*}
and
\begin{align*}
h(x+y) - h(x) - \langle y,\nabla h(x)\rangle\leq\Vert y\Vert^2/2,
\end{align*}
we have
\begin{align*}
\int_{\mathbb{R}^d}\Big(h(X_1(t)+x) - h(X_1(t)) - \mathbb{I}_{\Vert x\Vert<1}\langle x,\nabla h(X_1(t))\rangle\Big)&\nu(\text{d}x)\\
\leq&\frac{1}{2\beta^{2/\alpha}}\int_{\Vert x\Vert<1}\Vert x\Vert^2\nu(\text{d}x) + \frac{1}{2\beta^{1/\alpha}}\int_{\Vert x\Vert\geq1}\Vert x\Vert\nu(\text{d}x)\\
\leq&C\frac{d}{\beta^{1/\alpha}},
\end{align*}
where the last inequality is due to Lemma~\ref{lemma:levyMeasure}. By integrating \eqref{eqn:ItoFormulaForH} and combining the above inequalities, we have
\begin{align*}
h(X_1(t))-h(X_1(0))\leq& \int_0^t\Big((-mh(X_1(s))^{\gamma}+m+b)/2 + C\frac{d}{\beta^{1/\alpha}}\Big)\text{d}s + M(t)\\
\leq& \int_0^t\Big((m+b)/2 + C\frac{d}{\beta^{1/\alpha}}\Big)\text{d}s + M(t).
\end{align*}
By Lemma 3.8 in \cite{xie2017ergodicity}, for $\lambda\in(0,1)$,
\begin{align*}
\mathbb{E}\Big(\sup_{s\in[0,t]}h(X_1(s))^{\lambda}\Big)\leq c_\lambda\Big(\mathbb{E}h(X_1(0)) + ((m+b)/2 + C\frac{d}{\beta^{1/\alpha}})t\Big)^{\lambda}.
\end{align*}
This leads to the conclusion since $h(x)\geq\Vert x\Vert$.
\end{proof}

\begin{lemma}\label{lemma:expectationBoundOfX3}
For $\lambda\in(0,1)$, there exists a constant $c_b$ depending on $L,m,b,\alpha$, such that
\begin{align*}
\mathbb{E}\Big( \Vert X_3(t)\Vert^{\lambda}\Big)^{\frac{1}{\lambda}}\leq c_b\Big(t(d\beta^{-1/\alpha}+1)+1\Big),&& \forall t>0,\beta\geq1, 1<\alpha<2.
\end{align*}
\end{lemma}
\begin{proof}
The proof is similar to the proof of Lemma~\ref{lemma:expectationBoundOfX1}.
\end{proof}

\begin{lemma}\label{lemma:momentsOfStableDist}
Let $X$ be a scalar symmetric $\alpha$-stable distribution with $\alpha<2$, i. e. $X\sim \sas(1)$ (see Definition~\ref{def:StableDistribution}), then, for $-1<\lambda<\alpha$,
\begin{align*}
\mathbb{E}(\vert X\vert^{\lambda}) = \frac{2^{\lambda}\Gamma((1+\lambda)/2)\Gamma(1-\lambda/\alpha)}{\Gamma(1/2)\Gamma(1-\lambda/2)}.
\end{align*}
\end{lemma}
\begin{proof}
The proof follows from Theorem 3 in \cite{shanbhag1977certain} (see also equation (13) in \cite{matsui2016fractional}).
\end{proof}

\begin{corollary}
The quantity $l_{\alpha,\lambda,d}\triangleq\mathbb{E}\Vert L^{\alpha}(1)\Vert^{\lambda}$ is finite for $0\leq\lambda<\alpha$. For details, we have

(a) If $1<\lambda<\alpha$, then
\begin{align*}
\mathbb{E}\Vert L^{\alpha}(1)\Vert^{\lambda}\leq d^{\lambda}\Big(\frac{2^{\lambda}\Gamma((1+\lambda)/2)\Gamma(1-\lambda/\alpha)}{\Gamma(1/2)\Gamma(1-\lambda/2)}\Big).
\end{align*}

(b) If $0\leq\lambda\leq 1$, then
\begin{align*}
\mathbb{E}\Vert L^{\alpha}(1)\Vert^{\lambda}\leq d\Big(\frac{2^{\lambda}\Gamma((1+\lambda)/2)\Gamma(1-\lambda/\alpha)}{\Gamma(1/2)\Gamma(1-\lambda/2)}\Big).
\end{align*}

\end{corollary}

\begin{proof}
Since $L^{\alpha}(1)$, by definition, is a d-dimensional vector whose components are i.i.d symmetric $\alpha$-stable distributions $L_i^{\alpha}(1)$ for $i\in\{1,\ldots,d\}$, we have
\begin{align*}
\Vert L^{\alpha}(1)\Vert\leq&\sum_{i=1}^d \vert L_i^{\alpha}(1)\vert
\end{align*}

(a) $1<\lambda<\alpha$. By using Minkowski's inequality and Lemma~\ref{lemma:momentsOfStableDist},
\begin{align*}
(\mathbb{E}\Vert L^{\alpha}(1)\Vert^{\lambda})^{1/\lambda}\leq&\Big(\mathbb{E}\Big[\Big(\sum_{i=1}^d \vert L_i^{\alpha}(1)\vert\Big)^{\lambda}\Big]\Big)^{1/\lambda}\\
\leq&\sum_{i=1}^d (\mathbb{E}\vert L_i^{\alpha}(1)\vert^{\lambda})^{1/\lambda}\\
=&d\Big(\frac{2^{\lambda}\Gamma((1+\lambda)/2)\Gamma(1-\lambda/\alpha)}{\Gamma(1/2)\Gamma(1-\lambda/2)}\Big)^{1/\lambda}.
\end{align*}
Thus, we have
\begin{align*}
\mathbb{E}\Vert L^{\alpha}(1)\Vert^{\lambda}\leq d^{\lambda}\Big(\frac{2^{\lambda}\Gamma((1+\lambda)/2)\Gamma(1-\lambda/\alpha)}{\Gamma(1/2)\Gamma(1-\lambda/2)}\Big).
\end{align*}

(b) $0\leq\lambda\leq 1$. By using Lemma~\ref{lemma:anUsefulIneq} and Lemma~\ref{lemma:momentsOfStableDist} ,
\begin{align*}
\mathbb{E}\Vert L^{\alpha}(1)\Vert^{\lambda}\leq&\mathbb{E}\Big[\Big(\sum_{i=1}^d \vert L_i^{\alpha}(1)\vert\Big)^{\lambda}\Big]\\
\leq&\sum_{i=1}^d \mathbb{E}\vert L_i^{\alpha}(1)\vert^{\lambda}\\
=&d\Big(\frac{2^{\lambda}\Gamma((1+\lambda)/2)\Gamma(1-\lambda/\alpha)}{\Gamma(1/2)\Gamma(1-\lambda/2)}\Big).
\end{align*}

\end{proof}

\begin{lemma}\label{lemma:expecataionBound} Let us denote the value $\mathbb{E}\Vert L^{\alpha}(1)\Vert^{\lambda}$ by $l_{\alpha,\lambda,d}<\infty$. For $0<\eta\leq\frac{m}{M^2}$ and $s\in[j\eta,(j+1)\eta)$, we have the following estimates:

(a) If $1<\lambda<\alpha$ and $1<\gamma\lambda<\alpha$ then
\begin{align*}
&\mathbb{E}\Vert X_2(j\eta)\Vert^{\lambda}\leq B_{j,\lambda}\triangleq\Big( \Big(\mathbb{E}\Vert X_2(0)\Vert^{\lambda}\Big)^{\frac{1}{\lambda}} + j\Big((2\eta (b+m))^{\frac{1}{2}} +2^{\frac{1}{2}}\eta B +\Big(\frac{\eta}{\beta}\Big)^{\frac{1}{\alpha}}l_{\alpha,\lambda,d}^{\frac{1}{\lambda}}\Big)\Big)^{\lambda},\\
&\mathbb{E}\Vert X_2(s)\Vert^{\lambda}\leq \Big(B_{j,\lambda}^{\frac{1}{\lambda}} + (s-j\eta)\Big(M B_{j,\gamma\lambda}^{\frac{1}{\lambda}} + B\Big) + \Big(\frac{s-j\eta}{\beta}\Big)^{\frac{1}{\alpha}}l_{\alpha,\lambda,d}^{\frac{1}{\lambda}}\Big)^{\lambda}.
\end{align*}

(b) If $0\leq\lambda\leq 1$ then
\begin{align*}
&\mathbb{E}\Vert X_2(j\eta)\Vert^{\lambda}\leq \bar{B}_{j,\lambda}\triangleq\mathbb{E}\Vert X_2(0)\Vert^{\lambda} + j\Big((2\eta (b+m))^{\frac{\lambda}{2}} +2^{\frac{\lambda}{2}}(\eta B)^{\lambda} +\Big(\frac{\eta}{\beta}\Big)^{\frac{\lambda}{\alpha}}l_{\alpha,\lambda,d}\Big),\\
&\mathbb{E}\Vert X_2(s)\Vert^{\lambda}\leq \bar{B}_{j,\lambda} + (s-j\eta)^{\lambda}\Big(M^{\lambda} \bar{B}_{j,\gamma\lambda} + B^{\lambda}\Big) + \Big(\frac{s-j\eta}{\beta}\Big)^{\frac{\lambda}{\alpha}}l_{\alpha,\lambda,d}.
\end{align*}

(c) If $1<\lambda<\alpha$ and $0\leq\gamma\lambda\leq 1$ then
\begin{align*}
&\mathbb{E}\Vert X_2(j\eta)\Vert^{\lambda}\leq B_{j,\lambda},\\
&\mathbb{E}\Vert X_2(s)\Vert^{\lambda}\leq \Big(B_{j,\lambda}^{\frac{1}{\lambda}} + (s-j\eta)\Big(M \bar{B}_{j,\gamma\lambda}^{\frac{1}{\lambda}} + B\Big) + \Big(\frac{s-j\eta}{\beta}\Big)^{\frac{1}{\alpha}}l_{\alpha,\lambda,d}^{\frac{1}{\lambda}}\Big)^{\lambda}.
\end{align*}

\end{lemma}

\begin{proof}
Starting from
\begin{align*}
X_2((j+1)\eta) = X_2(j\eta) - \eta c_\alpha\nabla f(X_2(j\eta)) + \Big(\frac{\eta}{\beta}\Big)^{\frac{1}{\alpha}}L^{\alpha}(1),
\end{align*}
we have either (by Minkowski, if $\lambda>1$)
\begin{align}\label{eqn:lambdaGreatOne}
\Big(\mathbb{E}\Vert X_2((j+1)\eta)\Vert^\lambda\Big)^{\frac{1}{\lambda}} &\leq\Big(\mathbb{E}\Vert X_2(j\eta) - \eta c_\alpha\nabla f(X_2(j\eta))\Vert^{\lambda}\Big)^{\frac{1}{\lambda}} +\Big(\frac{\eta}{\beta}\Big)^{\frac{1}{\alpha}}\Big(\mathbb{E}\Vert L^{\alpha}(1)\Vert^{\lambda}\Big)^{\frac{1}{\lambda}},
\end{align}
or (by Lemma \ref{lemma:anUsefulIneq}, if $0\leq\lambda\leq 1$)
\begin{align}\label{eqn:lambdaLessOne}
\mathbb{E}\Vert X_2((j+1)\eta)\Vert^\lambda &\leq\mathbb{E}\Vert X_2(j\eta) - \eta c_\alpha\nabla f(X_2(j\eta))\Vert^{\lambda}+\Big(\frac{\eta}{\beta}\Big)^{\frac{\lambda}{\alpha}}\mathbb{E}\Vert L^{\alpha}(1)\Vert^{\lambda}.
\end{align}
We have
\begin{align}\label{eqn:boundOnGradDesc}
\nonumber\Vert X_2(j\eta) - \eta c_\alpha\nabla f(X_2(j\eta))\Vert^{\lambda} &= \Vert X_2(j\eta) - \eta c_\alpha\nabla f(X_2(j\eta))\Vert^{2\times\frac{\lambda}{2}}\\
\nonumber &=\Big(\Vert X_2(j\eta)\Vert^2 - 2\eta c_\alpha\langle X_2(j\eta),\nabla f(X_2(j\eta)\rangle +\eta^2\Vert c_\alpha\nabla f(X_2(j\eta)\Vert^2\Big)^{\frac{\lambda}{2}}\\
& \leq\Big(\Vert X_2(j\eta)\Vert^2 - 2\eta(m\Vert X_2(j\eta)\Vert^{1+\gamma}-b) +\eta^2(2M^2\Vert X_2(j\eta)\Vert^{2\gamma}+2B^2) \Big)^{\frac{\lambda}{2}},
\end{align}
where we have used assumption \Cref{assump:dissipative} and Lemma \ref{lemma:gradientBound}. For $0<\eta\leq\frac{m}{M^2}$,
\begin{align*}
2\eta m(\Vert X_2(j\eta)\Vert^{1+\gamma}+1)\geq 2\eta^2 M^2\Vert X_2(j\eta)\Vert^{2\gamma}.&& \text{(since $1+\gamma>2\gamma$ and $\eta m>\eta^2 M^2$)}
\end{align*}
Using this inequality we have
\begin{align}\label{eqn:boundOnGradDesc2}
\nonumber\Vert X_2(j\eta) - \eta c_\alpha\nabla f(X_2(j\eta))\Vert^{\lambda} &\leq\Big(\Vert X_2(j\eta)\Vert^2 + 2\eta (b+m) +2\eta^2 B^2 \Big)^{\frac{\lambda}{2}}\\
&\leq\Vert X_2(j\eta)\Vert^{\lambda} + (2\eta (b+m))^{\frac{\lambda}{2}} +2^{\frac{\lambda}{2}}(\eta B)^{\lambda}. &&\text{(by Lemma \ref{lemma:anUsefulIneq})}
\end{align}


Consider the case where $\lambda>1$. By \eqref{eqn:lambdaGreatOne} and \eqref{eqn:boundOnGradDesc2},
\begin{align*}
\Big(\mathbb{E}\Vert X_2((j+1)\eta)\Vert^\lambda\Big)^{\frac{1}{\lambda}} &\leq\Big(\mathbb{E}\Vert X_2(j\eta)\Vert^{\lambda} + (2\eta (b+m))^{\frac{\lambda}{2}} +2^{\frac{\lambda}{2}}(\eta B)^{\lambda}\Big)^{\frac{1}{\lambda}} +\Big(\frac{\eta}{\beta}\Big)^{\frac{1}{\alpha}}\Big(\mathbb{E}\Vert L^{\alpha}(1)\Vert^{\lambda}\Big)^{\frac{1}{\lambda}}\\
&\leq\Big(\mathbb{E}\Vert X_2(j\eta)\Vert^{\lambda}\Big)^{\frac{1}{\lambda}} + (2\eta (b+m))^{\frac{1}{2}} +2^{\frac{1}{2}}\eta B +\Big(\frac{\eta}{\beta}\Big)^{\frac{1}{\alpha}}l_{\alpha,\lambda,d}^{\frac{1}{\lambda}} \,\,\,\,\text{(by Lemma \ref{lemma:anUsefulIneq})}\\
&\leq\Big(\mathbb{E}\Vert X_2(0)\Vert^{\lambda}\Big)^{\frac{1}{\lambda}} + (j+1)\Big((2\eta (b+m))^{\frac{1}{2}} +2^{\frac{1}{2}}\eta B +\Big(\frac{\eta}{\beta}\Big)^{\frac{1}{\alpha}}l_{\alpha,\lambda,d}^{\frac{1}{\lambda}}\Big).
\end{align*}
For the case where $0\leq\lambda\leq 1$, by \eqref{eqn:lambdaLessOne} and \eqref{eqn:boundOnGradDesc2},
\begin{align*}
\mathbb{E}\Vert X_2((j+1)\eta)\Vert^\lambda &\leq\mathbb{E}\Vert X_2(j\eta)\Vert^{\lambda} + (2\eta (b+m))^{\frac{\lambda}{2}} +2^{\frac{\lambda}{2}}(\eta B)^{\lambda} +\Big(\frac{\eta}{\beta}\Big)^{\frac{\lambda}{\alpha}}l_{\alpha,\lambda,d}\\
&\leq\mathbb{E}\Vert X_2(0)\Vert^{\lambda} + (j+1)\Big((2\eta (b+m))^{\frac{\lambda}{2}} +2^{\frac{\lambda}{2}}(\eta B)^{\lambda} +\Big(\frac{\eta}{\beta}\Big)^{\frac{\lambda}{\alpha}}l_{\alpha,\lambda,d}\Big).
\end{align*}
Now, from the identification, for $s\in[j\eta,(j+1)\eta)$,
\begin{align*}
X_2(s)=X_2(j\eta) + (s-j\eta) c_\alpha\nabla f(X_2(j\eta)) + \Big(\frac{s-j\eta}{\beta}\Big)^{\frac{1}{\alpha}}L^{\alpha}(1),
\end{align*}
we have
\begin{align*}
\Vert X_2(s)\Vert &\leq\Vert X_2(j\eta)\Vert + (s-j\eta)c_\alpha\Vert\nabla f(X_2(j\eta))\Vert + \Big(\frac{s-j\eta}{\beta}\Big)^{\frac{1}{\alpha}}\Vert L^{\alpha}(1)\Vert\\
&\leq\Vert X_2(j\eta)\Vert + (s-j\eta)(M\Vert X_2(j\eta)\Vert^{\gamma} + B) + \Big(\frac{s-j\eta}{\beta}\Big)^{\frac{1}{\alpha}}\Vert L^{\alpha}(1)\Vert.
\end{align*}
For $\lambda >1$,
\begin{align*}
\Big(\mathbb{E}\Vert X_2(s)\Vert^{\lambda}\Big)^{\frac{1}{\lambda}} &\leq\Big(\mathbb{E}\Vert X_2(j\eta)\Vert^{\lambda}\Big)^{\frac{1}{\lambda}} + (s-j\eta)\Big(M\Big(\mathbb{E}\Vert X_2(j\eta)\Vert^{\gamma\lambda}\Big)^{\frac{1}{\lambda}} + B\Big) + \Big(\frac{s-j\eta}{\beta}\Big)^{\frac{1}{\alpha}}l_{\alpha,\lambda,d}^{\frac{1}{\lambda}}.
\end{align*}
For $\lambda\leq 1$,
\begin{align*}
\mathbb{E}\Vert X_2(s)\Vert^{\lambda} &\leq\mathbb{E}\Vert X_2(j\eta)\Vert^{\lambda} + (s-j\eta)^{\lambda}\Big(M^{\lambda}\mathbb{E}\Vert X_2(j\eta)\Vert^{\gamma\lambda} + B^{\lambda}\Big) + \Big(\frac{s-j\eta}{\beta}\Big)^{\frac{\lambda}{\alpha}}l_{\alpha,\lambda,d}.
\end{align*}
By replacing the estimate of $\mathbb{E}\Vert X_2(j\eta)\Vert^{\lambda}$, we obtain the desired result.

\end{proof}

\begin{lemma}\label{lemma:gradientBound} Under assumptions \Cref{assump:boundedGradAtZero} and \Cref{assump:HolderContinuity} we have
\begin{align*}
c_\alpha\Vert\nabla f(w)\Vert\leq M\Vert w\Vert^{\gamma} + B, && \forall w\in\mathbb{R}^d.
\end{align*}
\end{lemma}

\begin{proof}
By assumption \Cref{assump:HolderContinuity} we have 
\begin{align*}
c_\alpha\Vert\nabla f(w)-\nabla f(0)\Vert\leq M\Vert w-0\Vert^{\gamma}.
\end{align*}
Since $c_\alpha\Vert\nabla f(0)\Vert\leq B$ by assumption \Cref{assump:boundedGradAtZero}, the conclusion follows.
\end{proof}

\begin{lemma} For the function $b$ defined in Lemma~\ref{lemma:earlyLemma}, we have, for $w\in\mathbb{R}^d$,
\begin{align*}
&\Vert b(w)\Vert\leq M\Vert w\Vert^{\gamma} + (B+L),\\
&\langle w,b(w)\rangle\leq (L-m)\Vert w\Vert^{1+\gamma} + (b+L).
\end{align*}

\end{lemma}
\begin{proof}
From assumption \Cref{assump:uniformlyBounded}, it implies that
\begin{align*}
\Vert b(w)\Vert\leq c_\alpha\Vert\nabla f(w)\Vert+L.
\end{align*}
Then, by Lemma~\ref{lemma:gradientBound},
\begin{align*}
\Vert b(w)\Vert\leq M\Vert w\Vert^{\gamma} + (B+L).
\end{align*}
Next, by Cauchy-Schwarz inequality and assumption \Cref{assump:uniformlyBounded}, we have
\begin{align*}
\langle w,b(w)+c_\alpha\nabla f(w)\rangle\leq&\Vert w\Vert L.
\end{align*}
Then, by assumption \Cref{assump:dissipative},
\begin{align*}
\langle w,b(w)\rangle\leq& -c_\alpha\langle w,\nabla f(w)\rangle + \Vert w\Vert L\\
\leq& -m\Vert w\Vert^{1+\gamma} + b + \Vert w\Vert L\\
\leq& -m\Vert w\Vert^{1+\gamma} + b + (\Vert w\Vert^{1+\gamma}+1) L\\
=& (L-m)\Vert w\Vert^{1+\gamma} + (b + L).
\end{align*}
Here, we have used the inequality $\Vert w\Vert\leq\Vert w\Vert^{1+\gamma}+1$.
\end{proof}

\begin{lemma}\label{lemma:levyMeasure}
Let $\nu$ be the L\'evy measure of a $d$-dimensional L\'evy process $L^\alpha$ whose components are independent scalar symmetric $\alpha$-stable L\'evy processes $L_1^{\alpha},\ldots,L_d^{\alpha}$. Then there exists a constant $C>0$ such that the following inequality holds with $\beta\geq1$ and $2>\alpha>1$:
\begin{align*}
\frac{1}{\beta^{2/\alpha}}\int_{\Vert x\Vert<1}\Vert x\Vert^2\nu(\text{d}x) + \frac{1}{\beta^{1/\alpha}}\int_{\Vert x\Vert\geq1}\Vert x\Vert\nu(\text{d}x)\leq C\frac{d}{\beta^{1/\alpha}}.
\end{align*}
\end{lemma}

\begin{proof}
Using Lemma 4.1 in \cite{kallsen2006characterization}, we have
\begin{align*}
\int_{\Vert x\Vert<1}\Vert x\Vert^2\nu(\text{d}x)=&\sum_{i=1}^d \int_{\vert x_i\vert<1}\vert x_i\vert^2\frac{1}{\vert x_i\vert^{1+\alpha}}\text{d}x_i\\
=&\sum_{i=1}^d\frac{2}{2-\alpha}\\
=&\frac{2d}{2-\alpha}.
\end{align*}
Similarly, we have
\begin{align*}
\int_{\Vert x\Vert\geq1}\Vert x\Vert\nu(\text{d}x)=&\sum_{i=1}^d \int_{\vert x_i\vert\geq1}\vert x_i\vert\frac{1}{\vert x_i\vert^{1+\alpha}}\text{d}x_i\\
=&\sum_{i=1}^d\frac{2}{\alpha-1}\\
=&\frac{2d}{\alpha-1}.
\end{align*}
Combining these two equalities, we have the desired conclusion.
\end{proof}

\begin{lemma}\label{lemma:anUsefulIneq}
For $a,b\geq 0$ and $0\leq\gamma\leq 1$, we have the following inequality:
\begin{align*}
(a+b)^{\gamma}\leq a^{\gamma} + b^{\gamma}.
\end{align*}
\end{lemma}
\begin{proof}
If $a=b=0$, the inequality is trivial. Hence, let us assume that $a>b\geq 0$. We have
\begin{align*}
\Big(1+\frac{b}{a}\Big)^{\gamma}&\leq 1+\gamma \frac{b}{a} &&\text{(by Bernoulli's inequality)}\\
&\leq 1+\frac{b}{a} &&\text{(since $0\leq\gamma\leq 1$ and $\frac{b}{a}\geq 0$)}\\
&\leq 1+\Big(\frac{b}{a}\Big)^{\gamma}. &&\text{(since $0\leq\gamma\leq 1$ and $0\leq\frac{b}{a}\leq 1$)}
\end{align*}
By multiplying both sides by $a^{\gamma}>0$, we have the conclusion.
\end{proof}

\end{document}